\documentclass{amsart}
\usepackage{amssymb,latexsym,amscd}
\usepackage{pb-diagram}  
\renewcommand{\eqref}[1]{(\ref{#1})}   
\theoremstyle{plain}
\newtheorem{theorem}{Theorem}
\newtheorem{theorema}{Theorem}

\newtheorem{lemma}{Lemma}

\newtheorem{proposition}{Proposition}
\newtheorem{algorithm}{Algorithm}
\newtheorem{conjecture}{Conjecture}
\theoremstyle{definition}

\theoremstyle{remark}

\newcommand {\R}{{\mathbb{R}}}

\newcommand {\N}{{\mathbb{N}}}

\newcommand {\LL}{{\mathcal{L}}}

\newcommand {\eps}{{\epsilon}}

\begin{document}

\title{ On the greatest prime factor of $ab+1$  }
\date{\today}
\author{\'Etienne Fouvry}
\address{ Univ. Paris--Sud, Laboratoire  de  Math\' ematiques d'Orsay, CNRS, F-91405 Orsay Cedex, France}
\email{Etienne.Fouvry@math.u-psud.fr}
 

\begin{abstract}  We improve some  results on the size of the greatest prime factor of  the   integers of the form $ab+1$ where $a$ and $b$ belong to some general  given finite sequences $\mathcal A$ and $\mathcal B$ with rather large density.
\end{abstract}
\thanks{The author  benefited from the financial support of Institut Universitaire de France.}
 
\maketitle

\renewcommand{\theenumi}{(\roman{enumi})}
\section{Introduction}\label{intro}
  Let $N$ be an integer $\geqslant 1$ and  let $\mathcal A$ and $\mathcal B$ be two sets of integers, both included in $[1,N]$.   With these two finite sets, we build the set ${\mathcal C}={\mathcal C}({\mathcal A}, {\mathcal B})$, defined by
  $$
  {\mathcal C} = {\mathcal C}({\mathcal A}, {\mathcal B}) :=\bigl\{ab+1\, ;\ a\in {\mathcal A},\, b\in {\mathcal B}\bigr\}.
  $$
 Let $P^+ (n)$ be the greatest prime factor of the integer  $n$ if  $n\geqslant  2$ and $P(1)=1$. The object of this paper is to give a lower bound for the integer 
  $$
 \varGamma^+ ( {\mathcal A}, {\mathcal B}, N):=\max _{c\in{\mathcal C}}  P^+ (c),
 $$
in terms of $N$  and of the cardinalities
$ 
 \vert {\mathcal A}\vert $ and $\vert {\mathcal B}\vert.
$ 
 The interest of this question is  that we suppose no condition of regularity for the sets $\mathcal A$ and $\mathcal B$, but we   only impose some lower bound for $ \vert \mathcal A \vert $ and $\vert \mathcal B\vert $.  The purpose of the present paper is to improve the following
\begin{theorema} (See \cite[Theorem 2]{St}) For any positive $\eps$, there exist  positive  constants $c_1$, $c_2$ and $c_3$, depending at most on $\eps$, in an effective way,  such that, for any $N\geqslant  c_1$, for any subsets $\mathcal A$ and $\mathcal B$ of $[1,N]$, satisfying the inequalities
$$
 \vert  \mathcal A\vert,\ \vert \mathcal B \vert  \geqslant  c_2{N\over ((\log N) /\log \log N)^{1\over 2}},
$$
we have the inequality
\begin{equation}\label{1jan1}
  \varGamma^+ ( {\mathcal A}, {\mathcal B}, N)\geqslant  \min \Bigl\{ N^{1+(1-\eps)(\min ( \vert  \mathcal A\vert,\,  \vert \mathcal B \vert)/N)^2}, 
c_3(N/\log N)^{4\over 3}
\Bigr\}.
\end{equation}
\end{theorema}
It is important to note that, in the particular case where $\mathcal A$ and $\mathcal B$ are dense subsets of $[1,N]$, (which means that they satisfy $ \vert  \mathcal A\vert,\ \vert \mathcal B \vert\geqslant  \delta N$, for some fixed positive $\delta$ and $N\rightarrow +\infty$),   we then  have  
$\varGamma^+ ( {\mathcal A}, {\mathcal B}, N)\gg_{\delta} N^{1+\delta_1}$, where $\delta_1$ is a positive function of $\delta$. However the relation \eqref{1jan1} never produces a lower bound better than $N^{ 4\over 3}$.

Actually, much more is conjectured since in  \cite[Conj.1]{SaSt}, the authors propose the following 
\begin{conjecture}\label{21dec1} For every $\eps $ satisfying $0<\eps <1$, there exists $N(\eps)$ and $C(\eps ) >0$, such that, for every integer  $N\geqslant  N(\eps)$, for every $\mathcal A$ and  $\mathcal B \subset [1,\dots, N]$ satisfying
\begin{equation}\label{strict}
 \vert  \mathcal A\vert,\ \vert \mathcal B \vert> \eps N,
\end{equation}
we have the inequality 
\begin{equation}\label{14jan1}
 \varGamma^+ ( {\mathcal A}, {\mathcal B}, N) \geqslant  C(\eps ) N^2.
 \end{equation}
\end{conjecture}
 Such a conjecture   becomes  false if we only impose the lower bounds $\vert \mathcal A\vert$ and $\vert \mathcal B\vert \geqslant \eps (N)\cdot N$, where $\eps (N)$ is a function of $N$ tending to $0$ as slowly as we want, when $N$ tends to infinity.  To see this,  choose $p$ a prime  satisfying $(2\eps (N))^{-1} \leqslant  p\leqslant    \eps (N)^{-1}$  and consider $\mathcal A = 
 \{ a \leqslant N\, ;\ a\equiv 1 \bmod p\}$ and 
 $\mathcal B = 
 \{ b \leqslant N\, ;\ b\equiv -1 \bmod p\}$. For such $\mathcal A$ and $\mathcal B$, we easily see that $\varGamma^+ ( {\mathcal A}, {\mathcal B}, N) \leqslant (N^2+1)/p =o(N^2)$.

 Before stating our results, we first give some  general considerations  on the set $\mathcal C$.
 \subsection{The subset of Linnik--Vinogradov}  Let 
$$
{\mathcal {LV}}(N) :=\{n\, ; n\leqslant N^2, \, n=ab\text { with } 1\leqslant a, \, b\leqslant N\}.
$$
Hence $\mathcal {LV} (N)$ is the set of (distinct) products of  two integers $\leqslant N$. The study of the cardinality of this set is not an easy task at all, this a question due to Linnik and Vinogradov.    K.Ford \cite[Corollary 3]{Ford} has now solved this problem by proving     \begin{equation}\label{encadrement}
\vert \, \mathcal {LV} (N)\, \vert  \asymp {N^2\over (\log N)^{c_4} (\log \log N)^{3\over 2}},
\end{equation}
where $c_4$ has the value     $ c_4= 1 - {1+\log \log 2\over \log 2} =0.086\, 07 \dots$. (for a slightly weaker result see \cite[Theorem 23]{HaTe}).
The relation \eqref{encadrement} shows that $\mathcal{LV} (N)$ is  a sparse subset of $[1,N^2]$, but only by a tiny power of $\log N$.

Hence,     for any $\mathcal A$ and $\mathcal B$, we have the trivial relation
$$
 {\mathcal C}({\mathcal A}, {\mathcal B}) \subset \mathcal {LV}( N)+ \{ 1\},
 $$
 which shows in which sparse subset of $[1,N^2+1]$,  the set $\mathcal C$ lives obligatorily.
 
 To complete the description of the scenery of our problem,  we recall 
 the basic properties of the classical function $\Psi (x,y)$, which counts the integers less than $x$ with all their prime factors less than $y$. In other words, we define
 $$
 S(x,y):=\{ n\leqslant x\, ;\ P^+( n) \leqslant y\}
 $$
 and 
 $$
 \Psi (x,y):=\vert S(x,y) \vert.
 $$
 We only appeal to the rather easy result
 $$
 \Psi (x,y) =x \rho \Bigl( {\log x\over \log y}\Bigr) + O \Bigl( {x\over \log y}\Bigr),
 $$
 uniformly for $x\geqslant  y\geqslant 2$ (see \cite[Th\' eor\`eme 6, p.371]{Te}, for instance). Here $\rho$ is the Dickman function (see \cite[p.370]{Te}), This function quickly goes to zero, since it satisfies
 $$
 \rho (u) \leqslant 1/\Gamma (u+1), \ (u>0).
 $$
 
 Using the above formula, the Stirling formula, and the inclusion--exclusion principle, we see that
 $$
\Bigl( \mathcal {LV}(N)+\{1\}\Bigr) \cap \Bigl( [1,\cdots, N^2+1]\setminus S(N^2+1, y)\Bigr) \not= \emptyset,
 $$
 as soon as $N$ is sufficiently large and $y$ satisfies
 $$
 y\geqslant \exp \Bigl( c_5{\log N\, \log \log \log N\over \log \log N      }\Bigr),
 $$ 
 where $c_5$ is some absolute positive  constant.   This means that, with a naive approach, we proved that the shifted Linnik--Vinogradov set
 ${\mathcal {LV}}(N)+\{1\}$, contains an element divisible by a prime \begin{equation}\label{poorresult}
 p> N^{c_5{\log \log \log N\over \log \log N  }}. 
 \end{equation} 
 
 We    now state our results. They correspond to three different situations, which appear to be more and more difficult.
  We can already feel the depth   of   Conjecture \ref{21dec1} in the very particular case $ \mathcal A =\mathcal B= [1,\dots, N]$ (this corresponds to the condition \eqref{strict} with $\eps =1$, and $>$ replaced by $\geqslant $).  This very particular situation will be the object of Theorem \ref{L--V}.

\subsection{ The case $\mathcal A =\mathcal B =[1,\dots,N]$}
 Our first step will be to prove
 \begin{theorem}\label{L--V}
 For every $A >0$, there exists $N_0= N_0(A)$,  such that, for every $N\geqslant N_0 $,  the interval $[(1-(\log N)^{-A}) N^2, N^2]$ contains a prime $p$ of the form $p=ab+1$, where $a$ and $b$ are integers satisfying $1\leqslant a, b\leqslant N$.
 
 In particular, for $N\geqslant N_0 (A)$ we have the inequality 
 $$
 {\mathnormal \Gamma}^+(\mathcal A, \mathcal B, N) \geqslant  \Bigl( 1-{1\over (\log N)^A}\Bigr) N^2,
 $$
 under the constraints $\mathcal A =\mathcal B  =[1,\dots, N].$
 \end{theorem}
 Such a result has to be compared with the weak result given in \eqref{poorresult} and it  is far from being trivial by the tools which will be involved.   Theorem \ref{L--V} implies that 
 the inequality \eqref{14jan1} of Conjecture  \ref{21dec1} is true for any $C(\eps) <1$, in the particular case  $\mathcal A =\mathcal B  =[1,\dots, N].$  Its proof will be given in \S \ref{Thm1}, one of its qualities is to give a first idea of the difficulty of  the proof of  Conjecture \ref{21dec1}, if  
 such a proof exists. In our proof, we shall appeal to   the Siegel--Walfisz Theorem concerning primes in arithmetic progressions. This fact prevents to  produce an effective value for  $N_0(A)$, above. The same remark applies to Theorems  \ref{firstgeneral} and  \ref{secondgeneral}   below.
 
 A very delicate question is to find the asymptotic expansion  of the cardinality of the set of primes belonging 
 to ${\mathcal {LV}} (N)+\{ 1\}$. This question was treated in    \cite[Corollary 3]{InTi},  \cite{Ford} and finally in \cite[Corollary 1.1]{Kou} which gives the asymptotic order of magnitude of this cardinality.
  \subsection{ The case $\mathcal A=[1,\dots, N]$  and ${\mathcal B}$ general}  This is the second step in our graduation of difficulty. In \S \ref{Thm2}  we shall prove  
 
  \begin{theorem}\label{firstgeneral} 
  Let  $\delta$ satisfying $0<\delta <1$. 
  There exist  an absolute constant  $c_6$, independent of $\delta$,  and  a constant  $c_7=c_7 (\delta)$,  such that, for any $N\geqslant c_7$, for any subset $ \mathcal B$ of $[1,\dots, N]$ satisfying the inequality
 \begin{equation}\label{7jan1}
  \sum_{b\in \mathcal B \atop (1-\delta ) N < b \leqslant N} 1 \geqslant   {N\over \log^2 N} \cdot (\log \log N)^{c_6},
\end{equation}
  there is a prime $p$ in the interval $ ](1-2\delta) N^2  , (1-\delta) N^2  ]$ of the form $p=ab+1$ with $a$ and $b \leqslant N$ and $b\in \mathcal B$. 
  
   In particular, if $\mathcal A =[1,\dots , N]$ and if  $\mathcal B$ satisfies  
  \eqref{7jan1}, we have 
  $$
 {\mathnormal \Gamma}^+(\mathcal A, \mathcal B, N) \geqslant  (1-2\delta) \cdot N^2,
 $$
 for any sufficiently large $N$.
    \end{theorem}
   The condition \eqref{7jan1} is not artificial at all for the following reason:  the  order of magnitude of the prime number $p$ is almost $N^2$, hence, in the equality $p=ab+1$, both $a$ and $b$ must be  close to $N$. This implies that  the set $\mathcal B$ must contain many elements in the neighborhood of $N$.

 
 \subsection {The case $\mathcal A$ and $\mathcal B$ general} In the more general situation, we shall prove 
 
  \begin{theorem}\label{secondgeneral}
  For every    real number $0< \delta <1$, there exists  a function $\varpi_\delta\, :\  \N \to \R^*$ tending to zero at infinity, such that, for every $N\geqslant 2$,  for every subsets $\mathcal A$ and $\mathcal B$ of $[1,\dots,N]$, satisfying
  \begin{equation}\label{lowerbound}
 \vert  \mathcal A \vert \geqslant   \vert \mathcal B \vert \geqslant {N\over (\log  N)^{\delta}},
\end{equation}
  the following inequality 
   holds  
\begin{equation}\label{12:39}
\varGamma^+(\mathcal A, \mathcal B, N) \geqslant 
 N^{ 1+   ( \vert \mathcal A \vert / N )(1-\varpi_\delta (N))}.
 \end{equation}

   \end{theorem}
    It is possible to describe      the function $\varpi_\delta$ more precisely, but this  description will depend on non explicit constants  (see the comment after Theorem \ref{L--V}).
  Compared with \eqref{1jan1}, we see two advantages in Theorem \ref{secondgeneral}. When $\mathcal A$ is more and more dense, the exponent of $N$ in \eqref{12:39} tends to $2$. This gives some consistency to Conjecture \ref{21dec1}. 
  In the other direction, if $\mathcal A$ and $\mathcal B$ satisfy $\vert \mathcal A\vert \sim \vert  \mathcal B \vert \sim N/(\log N)^\delta$ with    $0<\delta <1/2$,  we obtain the lower bound
  $$
  \varGamma^+(\mathcal A, \mathcal B, N) \geqslant N^{1+(1-\eps) (\log N)^{-\delta}      },
  $$
 for $N> N_0 (\eps)$.  By \eqref{1jan1}, we would have have the same lower bound but with $\delta$ replaced s by $ 2\delta $, thus Theorem \ref{secondgeneral} represents a valuable improvement for sparse sequences.
 
 Actually, after a talk given by the  author at the Congress   {\it Activit\' es Additives et Analytiques}  (Lille, june  2009) where  he exposed the results of the present paper, C. Elsholtz kindly turned our attention on a preprint of K. Matom\" aki on the same subject. This work is now published (\cite{Ka}) from which we extract the following central result
 \begin{theorema}\label{kato}(\cite[Theorem 2]{Ka}
 Let $C_0$ and  $c_1$ be positive. Then for every $c_2$ satisfying
 $$
 c_2 < {1-4c_1 -{2\over C_0}\over 4},
 $$
 there exists $N_0 (c_2)$ such that, for every $N\geqslant N_0 (c_2)$, for every $\mathcal A$ and $\mathcal B\subset [1,\dots, N]$, satisfying
 $$
 \vert {\mathcal A}\vert \geqslant C_0{ N\over \log N} \ \text{ and} \ 
 \vert {\mathcal B}\vert \geqslant {\vert {\mathcal A}\vert \over N^{c_1 \vert {\mathcal A}\vert /N}},
 $$
 we have 
 $$ \varGamma^+ ( {\mathcal A}, {\mathcal B}, N) \geqslant N^{\sqrt{1 +c_2 \vert {\mathcal A} \vert/N}}. 
 $$
 \end{theorema}
 When writing the proof of Theorem \ref{kato}, the author was unaware of  \cite{St}, this is the reason why  she only refers to the older and weaker result  \cite{St0}. Hence it is worth comparing the strength of Theorems A \& \ref{kato}. Theorem \ref{kato} really takes into account the situation where  $\mathcal B$  is much thinner than $\mathcal A$ (a typical situation being $\vert {\mathcal A}\vert  \asymp  N$ and $\vert {\mathcal B}\vert \asymp  \vert 
 {\mathcal A}\vert \cdot  N^{-\delta}$ with $\delta >0$). In counterpart, Theorem \ref{kato} never produces a lower bound for $\varGamma^+ ( {\mathcal A}, {\mathcal B}, N)$ better than $\varGamma^+ ( {\mathcal A}, {\mathcal B}, N)\geqslant N^{\sqrt{1+{1\over 4}}-\eps}=N^{1.\, 118\cdots}$,   instead of  $\varGamma^+ ( {\mathcal A}, {\mathcal B}, N)\geqslant N^{1.\, 333\cdots}$ by Theorem A.
 
 It is also  worth noticing that in  \cite[\S 4] {Ka} the author expresses  the presentiment  of the importance of 
 the work of Bombieri, Friedlander and Iwaniec \cite{BFI2} to improve her results. The present paper confirms this intuition.  
 
 {\bf Acknowledgements.}  The author is  grateful to C. Elsholtz for letting him know the existence of \cite{Ka}. He also 
 warmly thanks C. Stewart for his stimulating conversations on the subject. 
  
 \section{Tools from analytic number theory}\label{Tools}

 \subsection{Primes in arithmetic progressions}
 In the rest of this paper, we reserve the letter $p$ to prime numbers. We shall also systematically write
 $$
 \mathcal L =\log 2x,
 $$
 where $x\geqslant 1$, is a real number we consider as tending to infinity.
 
  The proofs of   Theorems \ref{L--V},  \ref{firstgeneral} \& \ref{secondgeneral} are based on deep    properties     of the classical  function in prime number theory
 $$
 \pi (x;q,a) =\sum_{p\leqslant x\atop p\equiv a \bmod q} 1,
 $$
when $a$ and $q$ are coprime integers.  In particular,  its behavior has to be compared  with the function $\pi (x)/\varphi (q)$, where $\pi (x)$ is the cardinality of the set of primes $\leqslant x$ and $\varphi (q)$ is the Euler function of  the integer  $q$.
   We recall some properties of this counting function on average in arithmetic progressions. The most classical one is the Bombieri--Vinogradov Theorem  (see \cite{Bom0}, \cite{Vin},  \cite[Th\' eor\`eme 17]{Bom},  \cite[Theorem 17.1]{IwKo}, ... for instance)
 \begin{proposition}\label{B--V}  For every $A$, there exists $B=B(A)$ such that
\begin{equation}\label{9jan1}
 \sum_{q\leqslant Q}\  \max_{y\leqslant x}\  \max_{(a,q)=1} \ \Bigl\vert \pi( y;q,a)-{\pi (y)\over \varphi (q)}\Bigr\vert =O_A\bigl( x\, \mathcal L^{-A}\bigr),
\end{equation}
uniformly for  $Q\leqslant  x^{1\over 2}\, \mathcal L^{-B}$ and $x\geqslant 1$.
 \end{proposition}
 In many applications, this proposition replaces the Riemann Hypothesis 
 extended to Dirichlet's $L$--functions.  The best constant for the moment is $B=A+1$. But much more is conjectured: it is largely believed that \eqref{9jan1} is true for $Q=x^{1-\eps}$, for any $\eps >0$ (the $O$--constant depending now on $A$ and $\eps$). This is the content of  the Elliott--Halberstam Conjecture  (see \cite{ElHa}). The proof of the Bombieri--Vinogradov Theorem   (see \cite{IwKo}, for instance) is now presented as an elegant  and deep consequence of the large sieve inequality for multiplicative characters and of the combinatorial structure of the characteristic function of the set of primes or of the van Mangoldt function  $\Lambda (n)$ (see Lemma \ref{HeathB} below).  It is a challenge to improve the value of 
 $Q$ in \eqref{9jan1}, even by modifying the way of summing the error terms $\pi( y;q,a)-{\pi (y)\over \varphi (q)}$ or even by approaching the characteristic function of the set of primes by  the characteristic function of another set of the same, but easier, combinatorial structure. The first breakthrough 
 in that direction is  due to Fouvry and Iwaniec \cite{FoIw1} (see also \cite{FoTh})  and  it    was followed  by several papers of Bombieri, Fouvry, Friedlander and Iwaniec (\cite{Fo0}, \cite{Fo1}, \cite{FoIw2}, \cite{Fo3},   \cite {BFI1}, \cite{BFI2}, \cite{BFI3}...)  Also see \cite[\S 12 p.89--103]{Bom} for an introduction to these techniques, based on Linnik's dispersion method and on several types of bounds for  Kloosterman sums.
   
 For 
 the problem we are studying in the present paper, we shall restrict to two points of view. The first one is 
 \begin{proposition}\label{fouv1} (See \cite[Corollaire 1]{Fo2} \& \cite[Theorem 9]{BFI1})
For every non zero integer $a$,  for every $A$, we have the equality
 \begin{equation}\label{9jan2}
 \sum_{q\leqslant Q\atop (q,a)=1} \Bigl( \pi( x;q,a)-{\pi (x)\over \varphi (q)}\Bigr) =O_{a,A}\bigl( x\, \mathcal L^{-A}\bigr),
 \end{equation}
 uniformly for  $Q\leqslant  x\, \mathcal L^{- 200A -200}$ and $x\geqslant 1.$
 \end{proposition}
 Note that in \eqref{9jan2}, we are summing the error terms, without absolute value, on consecutive moduli $q$. Hence we benefit from oscillations of the signs of this error term. In the proof, this oscillation is exploited by {\it kloostermania} {\it i.e} by the  study of    sums of Kloosterman sums  with  consecutive denominators.     This is the heart of  the work of Deshouillers and Iwaniec \cite{DeIw}.

 To continue the presentation of our tools   we recall the classical   functions in prime number theory $\theta (x)$,  $\psi  (x)$ and $\psi (x;q,a)$ given by
 $$
 \theta (x) =\sum_{p\leqslant x} \log p,\ \psi (x) =\sum_{n\leqslant x} \Lambda (n)\text{ and } \psi(x; q,a) =\sum_{n\leqslant  x\atop n\equiv a \bmod q} \Lambda (n).
 $$
We also introduce the notation 
\begin{equation}\label{defsim}
q\sim Q 
\end{equation}
  to mean that $q$ satisfies the inequalities $Q\leqslant q< 2Q$.  
  
  In  the second variation in the thema of
  Proposition \ref{B--V} we sum the   error terms with absolute values giving 
  
   \begin{proposition}\label{Q=X^1/2} (see \cite[Main Theorem]{BFI2})
   There exists an absolute constant $B_1$ with the following property :
\vskip .2cm 
 For every integer  $a\not=0$,   for every $x$, $y$,   and $Q$ satisfying   $x\geqslant y\geqslant 3$, $Q^2\leqslant xy$, we have the inequality    
$$
\sum_{q\sim  Q \atop (q,a) =1}\ 
 \ \Bigl\vert \psi( x;q,a)-{\psi (x)\over \varphi (q)}\Bigr\vert\ 
 \ll x\,\Bigl( {\log y\over \log x}\Bigr)^2\cdot  (\log \log x)^{B_1}, 
$$
where the constant implied  in $\ll$ depends on $a$ at most.
 \end{proposition}
The trivial upper bound for the quantity studied in Proposition \ref{Q=X^1/2} is  $O(x)$. The same is also true for the quantities which are majorized in Propositions  \ref{B--V} and \ref{fouv1}.  Hence, when $Q=x^{1\over 2}$, the upper bound given in Proposition \ref{Q=X^1/2} is  non trivial only by a factor   $(\log x)^{-2} (\log \log x)^{B_1}$. It will be sufficient for our proof, however.
 
  Proposition \ref{Q=X^1/2} is also interesting for $Q=x^{{1\over 2}+\eps (x)}$, with $\eps (x)\to 0$ as $x\rightarrow \infty$, giving an asymptotic expansion of $\psi (x;q,a)\sim {\psi (x)\over \varphi (q)}$, for almost all  $q\sim Q$, satisfying $(q,a)=1$.  The technique of proof 
  of Proposition \ref{Q=X^1/2} was followed up  in \cite{BFI3}, leading to the relation $\psi (x;q,a)\asymp_\delta {\psi (x)\over \varphi (q)}$, for almost all $q\sim Q$, satisfying $(q,a)=1$, with $Q\leqslant x^{{1\over 2}+\delta}$ and where $\delta$ is a tiny positive constant.
 
 Actually, we shall use Proposition \ref{Q=X^1/2} under the form
\begin{equation}\label{2jan1}
\sum_{q\sim Q \atop (a,q) =1}\ 
 \ \Bigl\vert \pi( x;q,a)-{\pi (x)\over \varphi (q)}\Bigr\vert\ 
 \ll {x\over \log x}\cdot \Bigl( {\log y\over \log x}\Bigr)^2 \cdot  (\log \log x)^{B_1}. 
\end{equation}
 This is a standard consequence  of the inequality
$$
0\leqslant \psi (x) -\theta (x) \ll x^{1\over 2}\, \mathcal L,
$$
  and of the Abel summation formula written under the form
$$
\pi (x) = \int_{3\over 2}^x {1\over \log t} \  [{\rm d}\, \theta (t)],
$$
 and a similar formula for $\pi (x;q,a)$.
 The equality \eqref{2jan1} is well suited to the proof of Theorem \ref{firstgeneral} but is not sufficient for the proof of Theorem \ref{secondgeneral}.
    In \S \ref{goodday}, we shall adapt the original proof of  Proposition \ref{Q=X^1/2} to prove
 \begin{theorem}\label{extension} There exists an absolute constant $B_2$ with the following property :
 \vskip.2cm
 For every integer  $a\not=0$,    for every $x$, $y$, $P_1$, $P_2$ and $Q$ satisfying   $3\leqslant y\leqslant x$, $Q^2\leqslant xy$ and $1\leqslant P_1 \leqslant P_2$,  we  have
 the inequality 
\begin{align*}
 \sum_{q \sim Q  \atop (q,a)=1} \Bigl\vert 
\underset{P_1< p\leqslant P_2,\ pm\leqslant x\atop pm\equiv a \bmod q}{ \sum\ \ \  \sum} \log p  -{1\over \varphi (q)}
\underset{P_1< p\leqslant P_2,\ pm\leqslant x\atop (pm, q)=1}{ \sum\ \ \  \sum} \log p
 \Bigr\vert \ll   
  x\cdot {(\log y)^2\over \log x} \cdot  (\log \log x)^{B_2},
\end{align*}
 where the constant implied in $\ll$ depends   on $a$   at most.
 \end{theorem}
 Note that the trivial bound for the quantity  now studied   is 
\begin{align} 
& \sum_{q\sim Q} \Bigl( \sum_{n\equiv a \bmod q \atop n\leqslant x} \sum_{p\mid n} \log p +{1\over \varphi (q)} \sum_{(n,q)=1\atop n \leqslant x} \sum_{p\mid n} \log p\Bigr) \nonumber\\
&\ll   \sum_{q\sim Q}  \Bigl( \sum_{n\leqslant x \atop n\equiv a \bmod q}\log n +{1\over \varphi (q)} \sum_{n\leqslant x \atop (n,q)=1} \log n\Bigr)\nonumber \\
&\ll  x\, (\log 2x).\label{20jan1}
  \end{align}
  The proof  we shall give 
 is highly  based on the work
 \cite{BFI2}.   We shall use the same tools and, as far as possible the same notations,  but our proof is   more than a paraphrase of the original proof of \cite{BFI2}: the intrusion of the integer variable $m$ creates 
 a new combinatorial situation  that we cannot ignore. In the same order of ideas, the variable $m$ brings unsuspected difficulty linked with coprimality conditions (we  shall work a lot to circumvent the  condition $(A_3(x))$ below).  
  
 \section{Analytic preparation}\label{proofofTheorem4}
 We first recall some results concerning the average behavior of the divisor functions. The following subsection contains the results of Lemmas 11--15 of \cite{BFI2}.
 \subsection{Lemmas on divisor functions}\label{divisorfunctions}
  Let $\ell \geqslant 0$ be an integer, and let $n\geqslant 1$ be an integer, then we define
  $$
 \tau_\ell (n):=\sum_{n=n_1 \cdots n_\ell} 1,
  $$
  (this is the generalized divisor function of order $\ell$),  then    
  $$
  \tau (n) =\tau_2 (n),
  $$
 is the classical divisor function. Of course $\tau_0 (n)=1$ if and only if $n=1$, otherwise, its value is $0$. Again some notations:
 
\noindent $\bullet $ ${\mathbf 1}_{\mathcal E} $ is the characteristic function of the given subset  of integers $\mathcal E$,
 
\noindent $\bullet$  $\mathfrak z (n)$ is the characteristic function of the set of  integers $n$ divisible by no prime factor $<z$, where $z\geqslant  2$ is  a given number.  

   We first make a list of several upper bounds for sums of $\tau_\ell^k (n)$ and of $\mathfrak z (n) \tau_\ell (n)$. 
 \begin{lemma}\label{linnik1}
Let $k\geqslant 0$ and $\ell \geqslant 1$ be integers and let $\eps >0$.  We then have the inequality
 $$
 \sum_{x-y<n\leqslant x}\tau_\ell^k (n)\ll_\eps y\,  (\log 2x)^{\ell^k-1},
 $$
 uniformly for $x\geqslant y\geqslant x^\eps$ and $x\geqslant 1$.
 \end{lemma}
\begin{proof}
See \cite[Lemma 1.1.5]{Lin}  for instance or deduce  this lemma from the classical result of Shiu \cite{Shi}. 
\end{proof}
 It is well known that the main part of the divisor function $\tau_\ell (n)$  comes from the small divisors  of $n$. Hence the summatory functions of $\mathfrak z (n) \tau_\ell (n)$  and $\tau_\ell (n)$ have  different behaviors when $z$ becomes larger and larger.  This is the object of the next lemma.   
 \begin{lemma}\label{lemma13ofBFI2}
 Let $j\geqslant 0$ be an integer. We   have the six relations
 $$
 \sum_{n\leqslant x} \mathfrak z (n) \tau_j (n) \ll {x\over \log 2x}\cdot \Bigl( {\log 2xz\over \log 2z}\Bigr)^j ,\leqno (i) 
 $$
 $$
 \sum_{n\leqslant x} \mathfrak z (n) \tau_j (n)n^{-1} \ll     \Bigl( {\log 2xz\over \log 2z}\Bigr)^j ,\leqno (i')  
 $$
 $$
 \sum_{w<n\leqslant x} \mathfrak z (n) \tau_j (n)n^{-1}(\log 2n)^{-1}  \ll {1\over \log 2w}  \cdot \Bigl( {\log 2xz\over \log 2z}\Bigr)^j ,\leqno (i'')  
 $$
 $$
 \sum_{x<n\leqslant xy} \mathfrak z (n) \tau_j (n)n^{-1} \ll \Bigl({\log 2y\over \log 2x}\Bigr)\cdot \Bigl( {\log 2xyz\over \log 2z}\Bigr)^j, \leqno (ii) 
 $$
 $$
 \sum_{nt\leqslant x} \mathfrak z (n) \tau_{j  }(n) \ll  x\cdot (\log \log 3x)\cdot     \Bigl( {\log 2xz\over \log 2z}\Bigr)^{j } ,\leqno (iii)  
 $$
 and
 $$
 \sum_{x<nt\leqslant xy} \mathfrak z (n) \tau_{j } (n)(nt)^{-1} \ll (\log 2y)\cdot (\log \log 3xy) \cdot \Bigl( {\log 2xyz\over \log 2z}\Bigr)^{j }, \leqno (iv) 
 $$
 uniformly for  real   $w,\, x,\, y,\, z\geqslant 1$.
 \end{lemma}
 \begin{proof} For $j=0$, all  the results are trivial. For $j\geqslant 1$,  the items $(i)$ and $(ii)$ are exactly \cite[Lemma 13]{BFI2}. Also note that $(i)$ can also be seen as a  direct consequence of Shiu's result (\cite[Theorem 1]{Shi}) concerning sums of multiplicative functions,  with the adequate remarks concerning the uniformity of this result (see \cite[p.258]{Na} or \cite[p.119]{NaTe}).  
 
 The item $(i')$ is a direct consequence of Mertens formula. The inequality $(i'')$ is a trivial consequence of $(i')$.

 In the items $(iii)$ and $(iv)$, we impose no sifting condition on the variable $t$.  This explains the change in the asymptotic order.  We pass from $(i)$ to $(iii)$ by writing 
 $$
 \sum_{nt\leqslant x} \mathfrak z (n) \tau_{j } (n)=
 \sum_{t\leqslant x}  \ \ \sum_{  n\leqslant x /t}\  \mathfrak z (n) \tau_{j } (n) \ll
 \Bigl({\log 2xz\over \log 2z}\Bigr)^j
 \sum_{t\leqslant  x} {x/t\over \log (2x/t)} ,$$
 and summing over $t$.

   Finally for $(iv)$, we decompose 
   \begin{align}
    \sum_{x<nt\leqslant xy} &\mathfrak z (n) \tau_{j } (n)(nt)^{-1}\nonumber \\
  &=\Bigl\{
   \sum_{t\leqslant x} t^{-1}\  \sum_{x/t< n\leqslant xy/t} + \sum_{x<t\leqslant xy} t^{-1}\  \sum_{1< n\leqslant xy/t}
   \Bigr\}\mathfrak z (n)  \tau_j (n) n^{-1}\nonumber\\
  & :=\Sigma_1 +\Sigma_2,\label{20:56}
   \end{align}
 say.   For $\Sigma_1$ we use   $(ii)$ to write 
 the relations
 $$
 \Sigma_1 \ll ( \log 2y) \cdot \Bigl( {\log 2xyz\over \log 2z}\Bigr)^j \sum_{t\leqslant  x} {1\over t\log (2x/t)}
 \ll ( \log 2y) \cdot (\log \log 3x)\cdot  \Bigl( {\log 2xyz\over \log 2z}\Bigr)^j,
 $$
 which is acceptable in view of $(iv)$.
 Finally, for $\Sigma_2$, we use $(i')$  to write
 $$
 \sum_{1\leqslant  n \leqslant  xy/t} \mathfrak z (n) \tau_j (n) n^{-1} \ll  \Bigl( {\log 2xyz \over \log 2z}\Bigr)^j,
  $$
 for $t\geqslant x$. Inserting this bound into \eqref{20:56}, we complete the proof of $(iv)$.
  \end{proof}
 We continue our investigations for more intricate sums.  
 \begin{lemma}\label{22jan1}
  Let $j_1$, $j_2$, $j_3$ and $j_4$ be integers $ \geqslant  0$. We then have
\begin{align}
  \underset{\substack{n_1n_2n_3n_4\leqslant x \\
  w\leqslant n_4\leqslant n_3 \leqslant n_2 \leqslant n_1\\
   n_3\leqslant y n_4,\, n_1\leqslant y n_2\\
  }}{\sum\ \sum\ \sum\  \sum}&\ 
 \mathfrak z (n_1n_2n_3 n_4) \tau_{j_1} (n_1) \tau_{j_2} (n_2) \tau_{j_3} (n_3) \tau_{j_4} (n_4) \nonumber \\
&  \ll {x\over \log 2w}\, \Bigl( {\log 2y \over \log 2x}\Bigr)^2\, \Bigl( {\log 2xyz\over \log 2z}\Bigr)^{j_1+j_2+j_3+j_4},\label{11:11}
 \end{align}
 uniformly for $x$, $y$, $z$, $w\geqslant 1$. The constant implied in $\ll$ only depends on $j_1$, $j_2$, $j_3$ and $j_4$.
  Similarly we have  
 \begin{align}
  \underset{\substack{tn_1n_2n_3n_4\leqslant x \\
  w\leqslant n_4\leqslant n_3 \leqslant n_2 \leqslant t n_1\\
   n_3\leqslant y n_4,\, tn_1\leqslant y n_2\\
  }}{\sum\ \sum\ \sum\  \sum\  \sum} &\ 
 \mathfrak z (n_1n_2n_3 n_4)\,  \tau_{j_1 } (n_1) \tau_{j_2} (n_2) \tau_{j_3} (n_3) \tau_{j_4} (n_4) \nonumber \\
   \ll &(\log \log 3xyz)\cdot  {x\over \log 2w}\cdot  {(\log 2y )^2\over \log 2x}  \cdot  \Bigl( {\log 2xyz\over \log 2z}\Bigr)^{j_1+j_2+j_3+j_4 }   .\label{11:09}
 \end{align}
 Finally, the relation \eqref{11:09} remains true if the summation is replaced by each of the three following ones
 \begin{equation}\label{13:50}
  \underset{\substack{tn_1n_2n_3n_4\leqslant x \\
  w\leqslant n_4\leqslant n_3 \leqslant tn_2 \leqslant   n_1\\
   n_3\leqslant y n_4,\,  n_1\leqslant ty n_2 }}
 {\sum\ \sum\ \sum\  \sum\ \sum},\ \ 
  \underset{\substack{tn_1n_2n_3n_4\leqslant x \\
  w\leqslant n_4\leqslant tn_3 \leqslant n_2 \leqslant   n_1\\
   tn_3\leqslant y n_4,\,  n_1\leqslant y n_2\\
  }}{\sum\ \sum\ \sum\  \sum\ \sum} 
  \text{ or }
   \underset{\substack{tn_1n_2n_3n_4\leqslant x \\
  w\leqslant tn_4\leqslant n_3 \leqslant n_2 \leqslant   n_1\\
   n_3\leqslant  ty n_4,\,  n_1\leqslant y n_2\\
  }}{\sum\ \sum\ \sum\  \sum\ \sum}.
 \end{equation}
 \end{lemma}
\begin{proof} The upper bound \eqref{11:11} is exactly \cite[Lemma 14]{BFI2}. Remark that in \eqref{11:09} \& \eqref{13:50},  we are dealing with  sums  in dimension five since  we have replaced the variable $n_i$  in \eqref{11:11} by $tn_i$.  In that case, we say that 
the variable $t$ {\it is glued to $n_i$}.  This extra variable $t$, without sifting conditions, explains why the upperbound in \eqref{11:09} is larger than the corresponding one in \eqref{11:11} by a $\log 2x$--factor.  In our application
the value of  the exponent  of the $\log \log$--factor has no importance.   It remains to    adapt the proof of \cite[Lemma 14]{BFI2} to obtain \eqref{11:09} by appealing to  Lemma  \ref{lemma13ofBFI2}  and the upper bound   \eqref{11:11} of Lemma \ref{22jan1}.

We now give all the details for the proof of \eqref{11:09}, which corresponds to the case where $t$ is glued to $n_1$. 
By dyadic subdivision, we restrict   the summation to $x/2<tn_1n_2n_3n_4 \leqslant x.$ Playing with the 
conditions of summation in the left part of \eqref{11:09}, we deduce that the variables $n_2$, $n_3$ and $n_4$ satisfy
\begin{equation}\label{10:50}
n_2n_3n_4 \leqslant x^{3\over 4},\ n_3n_4\leqslant x^{1\over 2},\ w\leqslant n_4 \leqslant x^{1\over 4} \text{ and } x/2y <n_2^2n_3n_4\leqslant x.
\end{equation}
We first assume that
\begin{equation}\label{10:35}
y\leqslant x^{1\over 3}.
\end{equation}
We first sum on $t$ and $n_1$, by using Lemma \ref{lemma13ofBFI2} $(iii)$
\begin{equation}\label{name1}
\underset{t,\ \ n_1\atop tn_1\leqslant x/(n_2n_3n_4)}{\sum\ \  \sum}\  \mathfrak z (n_1)\,  \tau_{j_1}(n_1) \ll {x\over n_2n_3n_4} \cdot (\log \log 3xyz)\cdot \Bigl({\log 2xyz\over \log 2z}\Bigr)^{j_1}.
\end{equation}
Then, by Lemma \ref{lemma13ofBFI2} $(ii)$, by \eqref{10:50} and the restriction \eqref{10:35},  we have
\begin{equation}\label{name2}
\sum_{\sqrt{ x/2yn_3n_4}<  n_2\leqslant \sqrt{ x/n_3n_4}} \mathfrak z (n_2) \tau_{j_2} (n_2) n_2^{-1}
\ll {\log 2y \over \log 2x }\cdot  \Bigl({\log 2xyz\over \log 2z}\Bigr)^{j_2},
\end{equation}
\begin{equation}\label{12:09}
\sum_{n_4\leqslant n_3 \leqslant yn_4} \mathfrak z (n_3) \tau_{j_3} (n_3) n_3^{-1}
\ll {\log 2y \over \log 2n_4 }\cdot  \Bigl({\log 2xyz\over \log 2z}\Bigr)^{j_3},
\end{equation}
and finally
\begin{equation}\label{12:07}
\sum_{w\leqslant n_4\leqslant x^{1\over 4}}  \mathfrak z (n_4)\tau_{j_4} (n_4) n_4^{-1} (\log 2n_4)^{-1}\ll {1\over \log 2w} \cdot \Bigl( {\log 2xyz \over \log 2z}\Bigr)^{j_4},
\end{equation}
by Lemma  \ref{lemma13ofBFI2} $(i'')$. Putting together  \eqref{name1}, \eqref{name2}, \eqref{12:09} and \eqref{12:07}, we obtain \eqref{11:09} in the case of \eqref{10:35}.

We now suppose 
\begin{equation}\label{10:40}
y>x^{1\over 3}.
\end{equation}
Since $y$ is large, we lose almost nothing in forgetting the conditions
$n_3\leqslant yn_4 $ and $tn_1\leqslant yn_2$.  The sum that we are studying is less or equal to
\begin{equation}\label{12:28}
  \sum_{n_4\leqslant x^{1\over 4}} \mathfrak z (n_4) \tau_{j_4}(n_4)
  \sum_{n_3\leqslant x^{1\over 3}} \mathfrak z (n_3) \tau_{j_3}(n_3)
  \sum_{n_2\leqslant x^{1\over 2}} \mathfrak z (n_2) \tau_{j_2}(n_2)
 \underset {tn_1\leqslant x/(n_2n_3n_4)}{\sum\ \ \sum} \mathfrak z (n_1) \tau_{j_1}(n_1).
\end{equation}
Applying Lemma \ref{lemma13ofBFI2} $(iii)$ and $(i')$ three times, we see that the above quantity is
\begin{equation}\label{19fev1}
\ll x\cdot (\log \log 3xyz)\cdot  \Bigl({\log 2xyz\over \log 2z}\Bigr)^{j_1+j_2+j_3+j_4},
\end{equation}
which is less than the bound claimed in \eqref{11:09}, because of \eqref{10:40}. This completes the proof of \eqref{11:09} in all the cases.

The first case of \eqref{13:50} concerns the situation where $t$ is glued to $n_2$. The inequalities \eqref{10:50} are changed into
\begin{equation}\label{10:53}
tn_2n_3n_4 \leqslant x^{3\over 4},\ n_3n_4\leqslant x^{1\over 2},\ w\leqslant n_4 \leqslant x^{1\over 4} \text{ and } x/2y <t^2n_2^2n_3n_4\leqslant x,
\end{equation}
and we suppose that \eqref{10:35} is satisfied. By using respectively the items $(i)$ and  $(iv)$     of Lemma \ref{lemma13ofBFI2} we  can write 
\begin{equation}\label{16:17}
\underset{  \ n_1\atop n_1\leqslant x/(tn_2n_3n_4)}{\  \sum}\  \mathfrak z (n_1)\,  \tau_{j_1}(n_1) \ll {x/( tn_2n_3n_4)  \over \log 2x } \cdot \Bigl({\log 2xyz\over \log 2z}\Bigr)^{j_1},
\end{equation}
 $$\underset{ t,\ n_2\atop 
 \sqrt{ x/2yn_3n_4}<  tn_2\leqslant \sqrt{ x/n_3n_4}}{\sum\ \sum} \mathfrak z (n_2) \tau_{j_2} (n_2) (tn_2)^{-1}
\ll (\log \log 3xyz)\cdot (\log 2y )\cdot  \Bigl({\log 2xyz\over \log 2z}\Bigr)^{j_2},
$$
 and we use   \eqref{12:09} and  \eqref{12:07} again.
  Putting together the four  above results, we obtain \eqref{11:09}, for the first sum appearing in  
  \eqref{13:50} in the case where  \eqref{10:35} is satisfied. 
  
  By similar techniques we also prove that the second and third sums of \eqref{13:50} also satisfy \eqref{11:09} under the restriction \eqref{10:35}, which means $y$ small.  
  
  When \eqref{10:40} is satisfied ($y$ large) each of the three sums listed in \eqref{13:50} is less than the sum studied in \eqref{12:28}. By \eqref{19fev1}, we see that these three sums also satisfy \eqref{11:09}.
\end{proof}

Our last lemma on that subject is   
\begin{lemma}\label{15:10}
Let $x$, \,$y$, \,$z$,\, $w$ be real numbers $\geqslant 1$,  let $s=5$ or $6$ and  let $j_1,\dots , j_s$ be integers $\geqslant 0$.   We then have
\begin{equation}\label{15:12}\underset{
 \substack{n_1\cdots n_s\leqslant x\\
w\leqslant n_s\leqslant \cdots \leqslant n_1\\
n_{s-2} \leqslant yn_s}}{\sum\cdots \sum}
\mathfrak z (n_1\cdots n_s)\,  \tau_{j_1} (n_1)\cdots \tau_{j_s} (n_s)\ll
{x\over \log 2x} \, \Bigl( {\log 2y \over \log 2w}\Bigr)^2 \, \Bigl( {\log 2xyz \over \log 2z}\Bigr)^{j_1+\cdots + j_s},
\end{equation}
where the constant implied in $\ll$ depends at most on $j_1,\dots, j_s$. Similarly, we have for
   $s= 5$ or $6$ and $1\leqslant \nu \leqslant s$, the inequality
\begin{align}\label{13:00}
  \underset{(t,n_1,\dots, n_s)\in \mathcal E (s,\nu)}{\sum \cdots \sum} &
\mathfrak z (n_1\cdots n_s)\,  \tau_{j_1} (n_1)\cdots \tau_{j_s} (n_s)\\
&\ll x\cdot (\log \log 3xyz)^s \cdot 
   \Bigl( { \log 2y  \over \log 2w}\Bigr)^2 \cdot \Bigl( {\log 2xyz \over \log 2z}\Bigr)^{j_1+\cdots + j_s},\nonumber
\end{align}
where $\mathcal E (s,\nu)$ denotes the set of $s+1$--uples $(t,n_1,\dots, n_s)$ satisfying 
the inequalities
\begin{equation}\label{inequaforE} 
\begin{cases}
n_1\cdots (tn_\nu)\cdots n_s \leqslant x&\\
  w\leqslant n_s\cdots \leqslant (tn_\nu) \leqslant \cdots \leqslant n_1,&\\
n_{s-2}\leqslant yn_s,&\text { if }   \nu \not= s \text{ and }  s-2,\\
 tn_{s-2}\leqslant yn_s,&\text{ if } \nu =s-2,\\
n_{s-2} \leqslant  y tn_s,&\text{  if } \nu=s. 
\end{cases}
\end{equation}  
\end{lemma}
\begin{proof} Actually this lemma is also true for $s=4$, but we shall only use it in the cases $s=5$ or $s=6$ (see the end of \S \ref{boundaryconf}).
The upper bound \eqref{15:12} is exactly  \cite[Lemma 15]{BFI2}. The bound \eqref{13:00}  is a consequence of Lemma \ref{lemma13ofBFI2}.     Note that we pass from the conditions of summation of \eqref{15:12}  to $\mathcal E (s,\nu)$, by gluing (as we defined  after Lemma \ref{22jan1})  the variable $t$ to the variable $n_\nu$.

We now give the proof  of \eqref{13:00} in the particular case $s=\nu =5$ (in other words, this is the case  where  $t$ is glued to $n_5$)    since the other ten cases are similar.  We   write the inequality
\begin{align}
  \underset{(t,n_1,\dots, n_5)\in \mathcal E (5,5)}{\sum \cdots \sum}&\leqslant
  \sum_{w\leqslant tn_5\leqslant x^{1\over 5}} \mathfrak z (n_5) \tau_{j_5} (n_5) \sum_{tn_5\leqslant n_4\leqslant ytn_5} \mathfrak z (n_4) \tau_{j_4} (n_4)\label{sadio}\\
&\times \sum_{tn_5\leqslant n_3\leqslant ytn_5}\ \mathfrak z (n_3) \tau_{j_3} (n_3)
 \sum_{n_3\leqslant n_2\leqslant x^{4\over 5}/(n_3n_4tn_5)}\ \mathfrak z (n_2) \tau_{j_2} (n_2)\nonumber\\
& \times  \sum_{n_1\leqslant  x/(n_2n_3n_4tn_5)} \mathfrak z (n_1) \tau_{j_1} (n_1).\nonumber
 \end{align}
 By Lemma \ref{lemma13ofBFI2} $(i)$ we have
 \begin{equation}\label{sadio1}
  \sum_{n_1\leqslant  x/(n_2n_3n_4tn_5)} \mathfrak z (n_1) \tau_{j_1} (n_1) \ll {x/(n_2n_3n_4tn_5)\over \log  2x} \cdot  \Bigl({\log 2xyz\over \log 2z}\Bigr)^{j_1}.
  \end{equation}
  By Lemma  \ref{lemma13ofBFI2} $(i')$, we get
  \begin{equation}\label{sadio2}
 \sum_{n_3\leqslant n_2\leqslant x^{4\over 5}/(n_3n_4tn_5)}\ \mathfrak z (n_2) \tau_{j_2} (n_2) n_2^{-1} \ll  \Bigl({\log 2xyz\over \log 2z}\Bigr)^{j_2}.
 \end{equation}
 By   Lemma \ref{lemma13ofBFI2} $(ii)$, we have, for $i=3$ or $4$, the inequality
 \begin{equation}\label{sadio3}
  \sum_{tn_5\leqslant n_i\leqslant ytn_5}\ \mathfrak z (n_i) \tau_{j_i} (n_i)\,n_i^{-1}\ll {\log 2y\over \log 2tn_5}\cdot \Bigl({\log 2xyz\over \log 2z} \Bigl)^{j_i},
 \end{equation}
  and finally
  \begin{equation}\label{sadio4}
    \sum_{w\leqslant tn_5\leqslant x^{1\over 5}} \mathfrak z (n_5) \tau_{j_5} (n_5)(tn_5)^{-1} (\log 2tn_5)^{-2} \ll {\log x\over (\log 2w)^2}\cdot (\log \log 3xyz)\cdot \Bigl( {\log 2xyz\over \log 2z}\Bigr)^{j_5},
  \end{equation}
   by   Lemma  \ref{lemma13ofBFI2} $(iv)$ and the   lower bound
   $\log( 2tn_5)\gg \log 2w$.   Gathering  \eqref{sadio},...,\eqref{sadio4}, we deduce \eqref{13:00} in the particular case $(s,\nu)=(5,5)$.
   
   The other cases are treated similarly.
      \end{proof}

\subsection{Convolution  of  two sequences in arithmetic progressions.} \label{equidistribution}We continue to  follow the notations of \cite{BFI2}, in order to quote the necessary results from this paper. Let $f$ an arithmetic function with finite support.  We define
$$
\Vert f \Vert :=\Bigl( \sum_n \vert  f( n)\vert^2 \Bigr)^{1\over 2}.
$$
For $a$  and $q$ coprime integers, we introduce 
$$
\Delta (f;q,a):=\sum_{n\equiv a \bmod q} f( n) -{1\over \varphi (q)}\sum_{(n,q)=1} f(n).
$$
Hence $\Delta (f;q,a)$ measures the distribution of the sequence $f(n)$ in the arithmetic progression $n\equiv a \bmod q$. We shall be mainly concerned by the situation where $f$ is the   arithmetic convolution product  $f =\boldsymbol \alpha * \boldsymbol \beta$, of two  complex sequences $\boldsymbol \alpha=(\alpha_m)_  {m\sim M}$ and $\boldsymbol \beta=(\beta_n)_{n\sim N}$, with $MN =x$ say and $M,\, N > x^\eps$. (See \eqref{defsim} for the meaning of $\sim$). We shall also study  the convolution of three sequences. 

The following assumption 
is crucial in the context of dispersion technique. Let $B>0$ be a real number and $\kappa : \R\to \R$   a real function.  Now  consider the condition $(A_1 (B, \kappa)) $  concerning  $\boldsymbol \beta=(\beta_n)_{n\sim N}$

\vskip .1cm

\begin{equation*}
(A_1(B,\kappa))\ \ 
\begin{cases} \text{For any  } A>0, 
\text{ for any  integers } 
  d,\, k \geqslant 1, \ell \not= 0, (k,\ell)=1  \text{ we have }&\\ 
  \\ 
  \displaystyle \Bigl\vert \sum_{\scriptstyle n\equiv \ell \bmod k  \atop \scriptstyle (n,d) = 1} \beta_{n} -\frac{1}{\varphi(k)} \sum_{(n,dk)=1} \beta_{n}\Bigr\vert \leqslant \kappa (A)   \Vert  \boldsymbol \beta\Vert\, \tau^B(d) N^{1/2}(\log 2N)^{-A}.&\\
     \end{cases} 
\end{equation*}
\vskip .1cm   

 Of course any $(\beta_{n})_{{n\sim N}} $ satisfies $(A_{1}(B,\kappa))$ by chosing for $\kappa$ a huge function of $N$ and $A$ (for intance $\kappa (A) = (\log 2N)^A)$). This is an uninteresting case. 
The situation is quite different when 
  we deal with  sequences $(\beta_n)_{n\geqslant  1},$ which satisfy  Siegel--Walfisz  type 
  theorem (for instance the characteristic function of the set of primes). If, in that case,  we  consider 
  the truncated sequence  $\boldsymbol \beta=(\beta_n)_{n\sim N}$, then,   the condition  $A_1(B,\kappa))$ is satisfied by $\boldsymbol \beta=(\beta_n)_{n\sim N}$, but with a function  $A\mapsto \kappa (A)$ independent of $N$. Then we are in an interesting situation, on letting $N$ tend to infinity, and choosing $A$ very large, but fixed.

    We shall also frequently suppose that, on average,  the sequences are  less than a power of $\log 2n$
    by introducing, for $B>0$, the assumption
    \vskip .2cm
   $$
   \vert \beta_n\vert \leqslant B \,  \tau^B (n)\text{ for all } n\sim N.\leqno (A_2(B))
   $$
   \vskip .2cm
   Sometimes it will be asked that $\beta_n=0$ when $n$ has a small prime divisor in the following sense: let $x\geqslant  3$ be a real number and let  $A_3(x))$ be the hypothesis
   \vskip .2cm
   $$
   \beta_n \not=0 \Rightarrow \bigl\{ p\mid n \Rightarrow p  >\exp(\log x/(\log \log x)^2)\bigr\}.\leqno(A_3(x))
   $$
   \vskip .2cm
   In other words, we ask the support of $\boldsymbol \beta$ to be included in the set of quasi primes.
We shall also  sometimes work with very particular $\boldsymbol \lambda=(\lambda_\ell )_{\ell \sim L}$ satisfying 
   $$
   \text{ There exists an interval }\mathfrak L \subset [L, 2L[ \text{  and } z\geqslant  2 \text{ such that  }\boldsymbol \lambda =\mathfrak z {\mathbf 1}_{\mathfrak L}. \leqno (A_4 (z))
   $$
    
   First recall  a classical consequence of the large sieve inequality, which, after combinatorial preparations, leads  to   Proposition \ref{B--V} (Bombieri--Vinogradov Theorem).
   
   \begin{proposition}\label{largesieve}
   Let $\eps$, $x$, $B$, $M$ and $N$ be real numbers such that $\eps >0$, $B>0$, $x=MN$
and $M,\ N\geqslant \max (2, x^\eps)$.
Let $\kappa\,:\, \R \to \R$ be a real function.  Let $\boldsymbol \alpha =(\alpha_m)_{m\sim M}$, $\boldsymbol \beta =(\beta_n)_{n\sim N}$ be two complex  sequences  
such that $\boldsymbol \beta$ satisfies $(A_1(B,\kappa))$. Then, for every $C>0$, there exists $A_0$, depending only on $B$  and $C$ such that the following
    inequality holds
       $$
     \sum_{q\leqslant  x^{1\over 2}\mathcal L^{-A_0} }\ \max_{(a,q)=1}\, \bigl\vert \,\Delta (\boldsymbol \alpha *\boldsymbol \beta;q,a)\, \bigr\vert \ll \Vert \boldsymbol \alpha \Vert\,  \Vert \boldsymbol \beta \Vert\, x^{1\over 2} \, \mathcal L^{-C},\ 
        $$
 where   the constant implied in the $\ll$--symbol depends at most on $\eps$, $\kappa$,   $B$ and $C$. 
   \end{proposition}
 However, Proposition \ref{largesieve}  says nothing when $Q\asymp x^{1\over 2}$. We now
 recall several situations, when $Q$ ({\it level of distribution}) can be taken greater than $x^{1\over 2}$. The relative sizes of the factors of the convolution are crucial to allow to go beyond $x^{1\over 2}\mathcal L^{-A_0}$, which is the natural limit of the large sieve. 
 
 The first situation is

   \begin{proposition}\label{17:09} Let $a\not=0$ be an integer. Let $\eps$, $x$, $B$, $M$ and $N$ be real numbers such that $\eps >0$, $B>0$, $x=MN$
and $M,\ N\geqslant \max (2, x^\eps)$.
Let $\kappa\,:\, \R \to \R$ be a real function.
   Let $\boldsymbol \alpha =(\alpha_m)_{m\sim M}$, $\boldsymbol \beta =(\beta_n)_{n\sim N}$ be two  complex sequences such that  $\boldsymbol \beta$ satisfies $(A_1(B,\kappa))$, $(A_2(B))$ and $(A_3 (x))$.   
   
   Then for every $C>0$, we have
      $$
  \sum_{q\sim Q\atop (q,a)=1} \vert \Delta (\boldsymbol \alpha *\boldsymbol \beta;q,a)\vert \ll \Vert \boldsymbol \alpha \Vert\,  \Vert \boldsymbol \beta \Vert \,x^{1\over 2}\, \mathcal L^{-C}, 
   $$
   uniformly for
   $$
   x^{\eps -1} Q^2 < N < x^{{5\over 6}-\eps} Q^{-{4\over 3}},
   $$
 where   the constant implied in the $\ll$--symbol depends at most on $\eps$, $\kappa$,  $a$, $B$ and $C$.    
   \end{proposition}
The first and stronger  version of Proposition \ref{17:09} can be found  in  \cite[Th\' eor\`eme 1]{Fo1} (without  the restriction $(A_3(x))$). A new proof is given  
in \cite[Theorem 3]{BFI1} and it appears again as \cite[Theorem 1]{BFI2}.  It is obvious  that we can take $Q\asymp x^{1\over 2}$
as soon as $N$ satisfies $x^\eps < N < x^{{1\over 6}-\eps}.$ This result is quite convenient for applications.
  
We shall also use the following  result 
  which is one of the key ingredient in the proof of Proposition \ref{fouv1}. 
  \begin{proposition}\label{SAM1}  (\cite[Theorem 2]{BFI2})  Let $a\not= 0$ be an integer.  
  Let $\eps$, $x$,  $y_1$, $y_2$, $B$, $C$,  $N$  and $Q$ be real numbers such that $\eps >0$, $B>0$, $C>0$,   $y_2>y_1 >0$, $x \geqslant 1$, $x^\eps \leqslant N \leqslant x^{{1\over 3}-\eps}$
  and $1\leqslant Q \leqslant x^{1-\eps}$.
Let $\kappa\,:\, \R \to \R$ be a real function.
   Let   $\boldsymbol \beta =(\beta_n)_{n\sim N}$ be a  complex sequence  such that  $\boldsymbol \beta$ satisfies $(A_1(B,\kappa))$  and $(A_3 (x))$.

  \noindent Then,  for every double sequence    $\boldsymbol \xi =\xi (\ell,m)$     of  complex numbers, we   have the inequality 
$$
\sum_{q\sim Q\atop (q,a) =1}\Bigl(\ 
\underset
{\substack{\ell m n\sim x,\ n\sim N\\
y_1< m/n<y_2\\
\ell mn \equiv a \bmod q}}
{\sum \sum \sum}
\  \xi (\ell, m) \beta_n -{1\over \varphi (q)}
\underset
{\substack{\ell m n\sim x,\ n\sim N\\
y_1< m/n<y_2\\
(\ell mn ,q)=1}}
{\sum \sum \sum}
\  \xi(\ell, m) \beta_n\ \Bigr)
\ll \Vert \boldsymbol \xi\Vert\ \Vert \boldsymbol \beta \Vert \ x^{1\over 2} \, \mathcal L^{-C},
$$
 where  the constant implied in $\ll$ depends at most on 
$\eps$, $\kappa$, $a$, $B$ and $C$, and where
$$
\Vert \boldsymbol \xi\Vert =\Bigl(\underset{x/2N\leqslant \ell m\leqslant 2x/N}{\sum\ \sum}\vert \xi (\ell, m)\vert^2\Bigr)^\frac{1}{2}.
$$
\end{proposition}
 Note that we are summing the error terms without absolute values, this is why the level of distribution  $Q$ can be taken so large. If we fix $y_1=0$ and $y_2 =\infty$ and  define $\boldsymbol \alpha$
 by the formula $\alpha_k =\sum_{ \ell m= k} \xi(\ell, m)$, Proposition \ref{SAM1} deals with the convolution $\boldsymbol \alpha * \boldsymbol \beta$.

\subsection{Convolution  of  three sequences in arithmetic progressions.}\label{convolof3}
The second type of  results concerns the convolution of three sequences 
\begin{equation}\label{12:08}
\begin{cases}
\boldsymbol \eta=(\eta_k)_{k\sim K},\   \boldsymbol \lambda =(\lambda_\ell)_{\ell \sim L},\   \boldsymbol \alpha =(\alpha_m)_{m\sim M}, \\
\\
 x=KLM,\ \mathcal L =\log 2x, \text{ with }  K,\, L,\,M\geqslant 1.
 \end{cases}
\end{equation}
We have
\begin{proposition}\label{trilinear}
Let $\eps $, $B$ and $C$ be given positive real numbers. Let $a\not= 0$ be an integer. Let $\kappa\,:\,\R\to \R$ be a real function. Let $x$, $K$, $L$, $M$ be real numbers and   $\boldsymbol \eta$, $\boldsymbol \lambda $ and $\boldsymbol \alpha$ be three sequences as in \eqref{12:08}.  Furthermore, suppose that the  following  conditions are satisfied
\vskip .2cm
\noindent  $\bullet$ $K$, $L$, $M\geqslant  x^\eps$,
\vskip .2cm
\noindent  $\bullet$ $\boldsymbol \eta$ satisfies $(A_2(B))$ and $(A_3(x))$,
\vskip .2cm
\noindent  $\bullet$ $\boldsymbol \lambda$ satisfies $(A_1(B,\kappa))$, $(A_2(B))$ and $(A_3(x))$,
 \vskip .2cm
\noindent  $\bullet$ $\boldsymbol \alpha$ satisfies   $(A_2(B))$.
\vskip .2cm
Then, there exists $A_0$, depending only on $B$ and $C$, such that  the  following inequality
\begin{equation}\label{15:24}
\sum_{q\sim Q\atop (q,a)=1}\  \bigl\vert 
\Delta (\boldsymbol \eta * \boldsymbol \lambda *\boldsymbol \alpha; q,a)
\bigr\vert \ll x \mathcal L^{-C}.
\end{equation}
holds as soon as one of two sets of inequalities is verified
$$
Q\mathcal L^{A_0} < KL,\ K^2L^3 < Qx \mathcal L^{-A_0}\text{ and } K^4L^2 (K+L) < x^{2-\eps},
\leqno (S1)
$$
or 
$$
Q\mathcal L^{A_0} < KL,\ K L^2Q^2 < x^2 \mathcal L^{-A_0}\text{ and } K^2x^\eps< Q.
\leqno (S2)
$$
The constant implied in the $\ll$--symbol of  \eqref{15:24} depends at most on $\eps$,  $\kappa$, $a$, $B$ and~$C$.
\end{proposition} 
The conditions $(S1)$ correspond to  \cite[Theorem 3]{BFI2}, and the set $(S2)$ to \cite[Theorem 4]{BFI2}. Note that in the original statement of \cite[Theorems 3 \& 4]{BFI2}, the sequence $\boldsymbol \alpha$ is supposed to satisfy $(A_3(x))$.  Actually, this restriction is unnecessary, since the proof of \cite[Formula (4.3)]{BFI2}, based on Cauchy--Schwarz inequality  does not require such a condition.

Note that if in $(S1)$ or $(S2)$ the factor $\mathcal L^{A_0}$ was replaced by the larger factor $x^{\eps }$,    Proposition \ref{trilinear} would be too weak for the proof of Proposition \ref{Q=X^1/2} and Theorem \ref{extension}. This is the reason why we cannot appeal to \cite[Theorem 4]{BFI1}, which also deals with the convolution of three sequences.

\subsection{Other types of results on the convolution of   three sequences.   }\label{convolof3fouvry}  The condition $(A_3(x))$ that must satisfy  $\boldsymbol \lambda$ in Proposition \ref{trilinear} is rather annoying in the application that we have in mind.  It could certainly be removed by  writing with great care the original proof of Theorems 3 \& 4 of \cite{BFI2}.  We prefer to modify the proof of the     following  result of Fouvry \cite[Th\' eor\`eme 2]{Fo0}.
\begin{proposition}\label{fouvryoriginal}  Let $a$ be an integer.  Let $\kappa\,:\, \R \to \R$ be a real function. Let $\eps$,   $x$, $C$,  $L$,
$M$, $N$ be real numbers such that : $\eps$ and $C>0$, $L$, $M$ and $N\geqslant 1$, $x=LMN$,  $0<\vert a \vert \leqslant x$ and such that
$$
L^2N\leqslant M^{2-\eps},\ L^3N^4 \leqslant M^{4-\eps} \text{ and } \log N \geqslant \eps \log M.\leqno(S3)
$$
Let $\boldsymbol \alpha$, $\boldsymbol \beta$ and $\boldsymbol \lambda$ be the characteristic functions of three sets of integers respectively included in $[M, 2M[$, $[N,2N[ $ and $[L,2L[$. Suppose that
$\boldsymbol \beta$ satisfies
$$
\Bigl\vert \sum_{n\equiv b \bmod q}\beta_n -{1\over \varphi (q)} \sum_{(n,q)=1} \beta_n \Bigr\vert \leqslant  \kappa (A) \Bigl( \sum_n \vert\,\beta_n\,\vert \Bigr) \, (\log 2N)^{-A},
$$
for every  real $A$ and for every integers $b$ and $q$ such that $(b,q)=1$.
Then we have the inequality
$$ \sum_{q\leqslant  (LN)^{1-\eps}\atop (q,a)=1} \ \Bigl\vert\,  \Delta(\boldsymbol \alpha *\boldsymbol \beta *\boldsymbol \lambda ; q,a)\, \Bigr\vert
 \ll  x\, \mathcal L^{-C},
 $$
 where the constant implied in the $\ll$--symbol depends at most on $\eps$, $\kappa$ and  $C$.     
\end{proposition}
It is worth to notice the large uniformity over $a$ compared with the results contained in  Propositions \ref{17:09}--\ref{trilinear}. This is due to the use of Weil's classical bound for Kloosterman sums instead of kloostermania. However we shall not use this uniformity here. Nevertheless the condition $(A_3(x))$ is now absent from the hypothesis,  but the range of summation  for $q$ is not satisfactory for our application. As in Proposition \ref{trilinear}, we would like to go up to $q\leqslant (LN) \, \mathcal L^{-A_0}.$
 
 We now give the improvement  of Proposition \ref{fouvryoriginal} necessary for our application.

\begin{proposition}\label{fouvrytrilinear}
 Let $a$ be an integer.  Let $\kappa\,:\, \R \to \R$ be a real function. Let $\eps$,   $x$, $C$,  $L$,
$M$, $N$ be real numbers such that : $\eps$ and $C>0$, $L$, $M$ and $N\geqslant x^\epsilon$, $x=LMN$, 
such that
$(S_3)$ is satisfied. Let $a$ be an integer  such that$0<\vert a \vert \leqslant x$.
  Let $\boldsymbol \alpha =(\alpha_m)_{m\sim M}$, $\boldsymbol \beta  =(\beta_{n})_{n\sim N}$ and $\boldsymbol\lambda  =(\lambda_\ell)_{\ell \sim L}$ be three sequences such that   
  
\smallskip
\noindent $\bullet$  $\boldsymbol \alpha$,  $\boldsymbol \beta,$ and $\boldsymbol \lambda $ satisfy $(A_2(B))$, 

\smallskip
\noindent $\bullet$ $\boldsymbol \beta$ satisfies $(A_1(B,\kappa))$.

\smallskip
Then  there exists $A_0$ depending only on $B$ and $C$ such that  we have
$$
\sum_{q\sim Q\atop (q,a)=1}
\Bigl\vert\,\Delta (\boldsymbol \alpha* \boldsymbol \beta * \boldsymbol \lambda  ;q,a)\, \Bigr\vert \ll  x \mathcal L^{-C },
$$
  for $Q\leqslant (LN)\,\mathcal L^{-A_0}$. The constant implied in the $\ll$--symbol depends at most on $\eps$, $\kappa$, $B$  and $C$.  
\end{proposition}
\begin{proof} When $Q\leqslant (LN)^{1-\eps}$,  the extension from Proposition \ref{fouvryoriginal} to Proposition \ref{fouvrytrilinear} is straightforward by following the proof of \cite[Th\'eor\`eme 2]{Fo0}.

  Hence we are left with the case
\begin{equation}\label{cond1}
(LN)^{1-\eps} < Q \leqslant (LN)\mathcal L^{-A_0}.
\end{equation}

  We shall follow the notations of  \cite{Fo0} the most possible, even if they are different sometimes from \cite{BFI2}.
Let 
$$\boldsymbol \gamma = \boldsymbol \beta  * \boldsymbol \lambda.$$
Hence $\boldsymbol \gamma =(\gamma_{k})_k$ has its support included in $[K,4K[$, with $K:=LN$.  Note that 
\begin{equation}\label{boundforgamma}
\vert \gamma_k \vert \leqslant B^2 \tau^{2B+1} (k),
\end{equation}
  by $(A_2(B))$. In the following proof, we shall denote by $B^*$ a constant depending only on the constant $B$ appearing  in the assumptions $(A_1(B,\kappa ))$ and  $(A_2(B))$. The value of $B^*$ may change at each time it appears.

Let  $E(Q)$ be the sum
$$
E(Q): =\sum_{q\sim Q\atop (q,a)=1}\sum_{(m,q)=1} \vert \alpha_m\vert \, \Bigl\vert  \sum_{ k \equiv a
\overline m \bmod q}
  \gamma_k-{1\over \varphi (q)} \sum_{( k,q)=1}   \gamma_k
\Bigr\vert. 
$$
(here $\overline m$ is the multiplicative inverse of $m \bmod q$.)
Obviously, $E(Q) $ satisfies the inequality
$$
\sum_{q\sim Q\atop (q,a)=1}
\Bigl\vert\,\Delta (\boldsymbol \alpha* \boldsymbol \beta * \boldsymbol \lambda  ;q,a)\, \Bigr\vert \leqslant  E(Q).
$$
By the Cauchy--Schwarz inequality, by the assumption $(A_2(B))$ for $\boldsymbol \alpha$ and by  inversion of  summation,
we get the inequality (see \cite[p.365]{Fo0})
\begin{equation}\label{12fev1}
E^2 (Q)\leqslant \Vert \boldsymbol \alpha \Vert \  Q\ D(Q)\ll M\, Q \,\mathcal L^{B^*} \, D(Q) ,
\end{equation}
where the dispersion $ D(Q)$ is 
\begin{equation}\label{12fev2}
D(Q) :=W(Q)-2V(Q)+U(Q),
\end{equation}
with 
$$
U(Q):=\sum_{q\sim Q\atop (q,a)=1} \ \sum_{m\sim M\atop (m,q)=1}\,  \Bigl( {1\over \varphi (q)} \sum_{(k,q)=1} \gamma_k \Bigr)^2,
$$
$$
V(Q):=\sum_{q\sim Q\atop (q,a)=1} \ \sum_{m\sim M\atop (m,q)=1}
\Bigl( \ \sum_{k_1\equiv a \overline m \bmod q} \gamma_{k_1}\Bigr) \Bigl( {1\over \varphi (q)} \sum_{(k_2,q)=1} \gamma_{k_2} \Bigr), 
$$
and
$$
W(Q):=\sum_{q\sim Q\atop (q,a)=1} \ \sum_{m\sim M\atop (m,q)=1} \Bigl(\    
\sum_{k\equiv a \overline m \bmod q} \gamma_k\Bigr)^2.
$$
Let 
also
$$
A(Q):= \sum_{q\sim Q\atop (q,a)=1} {1\over q\, \varphi (q)} \Bigl( \sum_{(k,q)=1} \gamma_k\Bigr)^2.
$$

Following the proof of \cite[Form.(6)]{Fo0} and appealing to Lemma \ref{linnik1},  we prove the equality
\begin{equation}\label{12fev3}
U(Q) =MA(Q) + O \bigl(   K^{2}\, Q^{-1}\, \mathcal L^{B^*}\bigr).
\end{equation}
By the proof of  \cite[Form.(12)]{Fo0}, we also have
\begin{equation}\label{12fev400}
V(Q)= M A(Q) + O_\eps\bigl( KQ^{-1} x^{1-\eps}+ K^{5\over 2} Q^{-1} x^{7\eps} \bigr),
\end{equation}
where  $\eps$ appears in \eqref{cond1}.

The study of $W (Q)$ is more delicate. Firstly we take some care to get rid of the common divisors.  Let
\begin{equation}\label{20:57}
\Delta := 3KQ^{-1},
\end{equation}
and, by \eqref{cond1},  we can suppose that
\begin{equation}\label{14fev00}
3\mathcal   L^{A_0}\leqslant \Delta < 3K^{\eps}.
\end{equation}
Then we notice that if $k_1$ and $k_2$ are two distinct integers of the interval $[K,4K[$,   satisfying $(k_1k_2, q)=1$ and $k_1- k_2 =q_0 q$, for some $q \sim Q$ and some positive  integer $q_0$
 we then have
\begin{equation}\label{12fev10}
 (k_1,k_2)= (k_1, k_1-k_2) =(k_1, q_0q) = (k_1, q_0) \leqslant  q_0\leqslant  \Delta.
\end{equation}
Following \cite[\S VI]{Fo0}, we write the equality
\begin{equation}\label{defW}
W(Q) = \sum_{q\sim Q\atop (q,a)=1}\ \sum_{k_1\equiv k_2\bmod q\atop (k_1k_2,q)=1}
\gamma_{k_1} \gamma_{k_2} \sum_{m\sim M \atop m\equiv a \overline k_1 \bmod q}  1.
\end{equation}
We first notice that the contribution, say $W^= (Q)$,  to $W(Q)$ of the $(k_1,k_2)$ with $k_1=k_2$
satisfies
\begin{align*}
\bigl\vert W^= (Q)\bigr\vert& \ll  \sum_{k\leqslant  4K} \tau^{4B+2} (k) \sum_{m\sim M\atop km\not=a} \tau (\vert km-a\vert ) +Qx^\eps .
\end{align*}
Writing $t=km$, we deduce that 
\begin{equation}\label{boundW=}
\bigl\vert W^= (Q)\bigr\vert  \ll \sum_{t\leqslant  8x\atop t\not= a}  \tau^{4B+3} (t) \tau( \vert t-a \vert ) +Qx^\eps \ll x\mathcal L^{B^*},
\end{equation}
by Cauchy--Schwarz inequality and by Lemma \ref{linnik1}.  We will see that  the bound \eqref{boundW=} is acceptable in view of  \eqref{cond1} \& \eqref{12fev1} by choosing $A_0$ sufficiently large.

Let  $W^{\not=} (Q)$ be the contribution to $W(Q)$ of the pairs  $(k_1,k_2)$ with $k_1\not= k_2$ (see \eqref{defW}).   By \eqref{12fev10}, we know that $d:=(k_1,k_2)$ is less than $\Delta$. Decomposing $W^{\not=} (Q)$ according to the the value of $d $ and writing $k_i=dk'_i$ ($i=1$, $2$)   we have  the equality (compare with \cite[Form.(13)]{Fo0})
\begin{equation}\label{20:37}
W^{\not=} (Q)= \sum_{q\sim Q \atop (q,a)=1}
 \sum_{d\leqslant \Delta\atop (d,q)=1}
\sum_{k'_1\equiv  k'_2 \bmod q, \, k'_1\not= k'_2\atop (k'_1,k'_2)= (k'_1k'_2, q) =1} \gamma_{dk'_1} \gamma_{dk'_2} \sum_{m\sim M \atop m\equiv a \overline{d k'_1} \bmod q} 1.
\end{equation}
We continue to prepare the variable by extracting from $k'_1$ all the prime factors 
appearing also in $d$. So we write $k'_1 =d_1 k''_1$ with $d_1 \mid d^\infty$ and $(k''_1, d)=1$ and we use the following crude estimate 
\begin{lemma}\label{easy}
Uniformly for $d$ integer $\geqslant 1$ and $y\geqslant 1$, we have the inequality
$$
\sum_{d_1 \mid d^\infty \atop d_1\geqslant y}{1\over d_1} \ll {\tau (d)\over y^{1\over 2}}.
$$
\end{lemma}
\begin{proof}  We may restrict to the case where $d=p_1\cdots p_r$ is squarefree. Following Rankin's method, we write, for every   $ \kappa \in ]0,1[ $,  the inequality
$$
\sum_{d_1 \mid d^\infty \atop d_1\geqslant y}{1\over d_1} \leqslant     
\sum_{d_1\mid d^\infty} {1\over d_1}\cdot \Bigl( {d_1\over y}\Bigr)^\kappa = {1\over y^{ \kappa}} \prod_{i=1}^r\Bigl( 1 -p_i^{\kappa -1}\Bigr)^{-1} \ll {1\over y^{ \kappa}}\exp\bigl( \sum_{i=1}^r p_i^{ \kappa -1}\bigr).
$$
Fixing $\kappa=1/2$, we get the desired upper bound.
\end{proof}
Inspired by \cite[p.368]{Fo0}, we see that the contribution to the right part
of \eqref{20:37} of $d_1>y$  is
$$=
\sum_{q\sim Q\atop (q,a)=1)} \sum_{d\leqslant \Delta\atop (d,a)=1} \sum_{d_1\mid d^\infty\atop d_1 >y}\  \sum_{\substack{d_1k''_1 \equiv k'_2\bmod q\\  d_1k''_1 \not= k'_2 \\ (k''_1, dk'_2)=(k''_1k'_2,d_1q)=1} } \gamma_{dd_1k''_1} \gamma_{dk'_2}\sum_{m\sim M\atop m\equiv a \overline{d d_1 k''_1}\bmod q }1. 
$$
Using \eqref{boundforgamma}  and the inequality $\tau (n)\ll X^{\eps\over 20 (2B+1)}$ ($0<n\leqslant X$ ) several times, 
and separating the cases $dd_{1}k''_{1}m-a\not =0$ from the case $dd_{1}k''_{1}m-a=0$, we see, by  
 \eqref{20:57}, that the above  contribution is
\begin{align}\label{radio}
&\ll M x^{\eps  \over 5}    \sum_{d\leqslant \Delta\atop (d,a)=1} \sum_{d_1\mid d^\infty\atop d_1 >y}\  \sum_{k''_1\leqslant 4K/(dd_1)} {K\over dQ} +Kx^\frac{\varepsilon}{10}\nonumber\\
&\ll  K^2MQ^{-1}x^{\eps\over 3} y^{-{1\over 2}}+Kx^\frac{\varepsilon}{10}\nonumber\\
&\ll K^2 M Q^{-1} x^{-{\eps \over 6}},
\end{align}
by choosing 
\begin{equation}\label{radio1}
y=x^{ \eps},
\end{equation}
 and applying Lemma \ref{easy}. Note that \eqref{radio} is acceptable in view of     \eqref{cond1} \& \eqref{12fev1}. Gathering \eqref{defW}, \eqref{boundW=}, \eqref{20:37} \& \eqref{radio} and slightly changing the notations,  we write 
 the equality
 \begin{align}\label{22fev1}
W  (Q)= &\sum_{q\sim Q \atop (q,a)=1}
 \sum_{d\leqslant \Delta\atop (d,q)=1} \sum_{d_1\mid d^\infty \atop d_1 \leqslant y}
\sum_{d_1k'_1\equiv  k_2 \bmod q, \, d_1k' _1\not= k _2\atop ( k'_1,dk_2)= (k'_1k_2, d_1q) =1} \gamma_{dd_1k'_1} \gamma_{dk_2} \sum_{m\sim M \atop m\equiv a \overline{dd_1 k'_1} \bmod q} 1\\
&+
O\bigl(  x\mathcal L^{B^*} + x^2 M^{-1} Q^{-1} x^{-{\eps\over 6}}\bigr).\nonumber
\end{align}
 (Compare with \cite[Form.(14)]{Fo0}). The main term in \eqref{22fev1} certainly comes in replacing the last sum by its approximation ${M/q}$.
 When this  replacement is done we can forget the conditions $d_1 k_1\not= k_2$ and $d_1\leqslant y$.  We introduce an error  which is in  $\ll x\mathcal L^{B^*} + K^2MQ^{-1} x^{-{\eps\over 6}}$ (same computations as for \eqref{boundW=}  \& \eqref{radio}). Let 
 \begin{equation}\label{defB}
 B(Q) := \sum_{q\sim Q\atop (q,a)=1} {1\over q} \ \sum_{k_1\equiv k_2\bmod q\atop (k_1k_2, q)=1} \gamma_{k_1}\gamma_{k_2}  
 =
 \sum_{q\sim Q\atop (q,a)=1} {1\over q} \ \sum_{\kappa\bmod q\atop (\kappa, q)=1} \Bigl(\  \sum_{k\equiv \kappa\bmod q}\gamma_{k}\Bigr)^2.
 \end{equation}
 The above discussion transforms \eqref{22fev1} into the following equality, which has to be compared with \cite[Form.(15)]{Fo0} 
 \begin{equation}\label{22fev2}
 W(Q) =M B(Q) +W_1 (M,Q) -W_1 (2M,Q) +O\bigl(  x\mathcal L^{B^*} + x^2 M^{-1} Q^{-1} x^{-{\eps\over 6}}\bigr),
 \end{equation} 
 with
 $$
 W_1 (Y,Q)= \sum_{q\sim Q \atop (q,a)=1}
 \sum_{d\leqslant \Delta\atop (d,q)=1} \sum_{d_1\mid d^\infty \atop d_1 \leqslant y}
\sum_{d_1k'_1\equiv  k_2 \bmod q, \, d_1k' _1\not= k _2\atop ( k'_1,dk_2)= (k'_1k_2, d_1q) =1} \gamma_{dd_1k'_1} \gamma_{dk_2} 
\psi \Bigl( {Y-a \overline{dd_1 k'_1}\over q}\Bigr),
 $$
where $\psi (t) +1/2$ is the fractional part of $t$.  Our present formula of  $W_1 (Y,Q)$ coincides  with the  corresponding formula of $W_1 (Y,Q)$ given in \cite[p.369]{Fo0}, with the tiny difference that the sum is over $d\leqslant x^\eps$ instead of $d\leqslant
\Delta$. In \cite{Fo0}, the problem of bounding $W_1(Y,Q)$ (with $Y=M$ or $2M$) is accomplished by appealing to Weil's bound for Kloosterman sums. It is easy to check, that, in this paper, the summation over $d$ is always made on the norms of the corresponding sums. Hence, since by \eqref{14fev00}, we have $\Delta\leqslant x^\eps$, we can apply \cite[Form.(26)]{Fo0}, in our case, giving the bound
\begin{equation}
\label{22fev3}
W_1(Y,Q) \ll L^2MN^2 Q^{-1} x^{-{\eps \over 2}} +L^3N^{5\over 2} Q^{-1} x^{3\eps} +L^{11\over 4} N^3 Q^{-1} x^{5\eps},
\end{equation}
for $Y=M$ or $2M$. By the orthogonality of characters,  the large sieve inequality and the  assumption $(A_1(B,\kappa))$ for $\boldsymbol \beta$, we get (compare with  \cite[Form.(37) \& (40)]{Fo0}) the inequality
\begin{align}\label{22fev4}
0\leqslant  MB(Q)-MA(Q) \leqslant  \sum_{q\sim Q\atop (q,a)=1} {1\over q\, \varphi (q)} \sum_{\chi \bmod q \atop \chi \not= \chi_0} \Bigr\vert \sum_k \chi (k)\gamma_k \Bigr\vert^2\nonumber\\ 
\ll x^2M^{-1} Q^{-1} \mathcal L^{B^*-2C} +x \mathcal L^{B^*},
\end{align}
which is true for any $C>0$. Gathering \eqref{12fev1}, \eqref{12fev2}, \eqref{12fev3}, \eqref{12fev400}, \eqref{22fev2}, \eqref{22fev3} \& \eqref{22fev4}, we can write
\begin{align*}
E^2 (Q)\ll MQ\mathcal L^{B^*}\Bigl\{ &L^2N^2 Q^{-1}+ LNQ^{-1} x^{1-\eps} + L^{5\over 2} N^{5\over 2}Q^{-1} x^{7\eps} +L^2MN^2 Q^{-1} x^{-{\eps \over 6}} \\
&+L^3N^{5\over 2} Q^{-1} x^{3\eps} +L^{11\over 4} N^3 Q^{-1} x^{5\eps}+  L^2MN^2  Q^{-1}  \mathcal L^{ -2C} +x
\Bigr\},
\end{align*}
which simplifies into
\begin{align*}
 E^2 (Q)\ll  x^2\mathcal L^{B^*-2C} +MQx\mathcal L^{B^*} +L^3 M   N^{5\over 2} x^{8\eps} +L^{11\over 4} M N^3 x^{8\eps}.
 \end{align*}
 This gives    bound claimed in Proposition \ref{fouvrytrilinear}, under the assumptions  $(S3)$ and \eqref{cond1} after    changing the value of $\eps$ and $C$.

\end{proof}

\subsection{Particular cases of equidistribution}\label{rabbit}
We now finish with some particular cases where $\boldsymbol \lambda$ is the characteristic function of quasi primes. The first result is \cite[Theorem 5*]{BFI2}.
\begin{proposition}\label{trilinearsmooth}
Let $a\not= 0$ be an integer. Let $\eps$, $z$,  $B$ and $C$ be positive numbers.    Let $x$,  $K$, $L$, $M$,  $\boldsymbol \eta$, $\boldsymbol \lambda $ and $\boldsymbol \alpha$  as in \eqref{12:08}, and satisfying the extra conditions 
\vskip .2cm
\noindent  $\bullet$ $1\leqslant z\leqslant \exp (\log 2x/(\log \log 2x))$,
\vskip .2cm
\noindent  $\bullet$ $K$, $L$, $M\geqslant  x^\eps$,
 \vskip .2 cm
\noindent  $\bullet$  $\boldsymbol \alpha$ and $\boldsymbol \eta$ satisfies $(A_2(B))$, 
\vskip .2cm
\noindent  $\bullet$ $\boldsymbol \lambda$  satisfies $(A_4 (z))$.
 \vskip .2cm
 Then there exists $A_0$ depending only on $B$ and $C$, such that the following  inequality holds
\begin{equation}\label{22fev10}
\sum_{q\sim Q\atop (q,a)=1}
\Bigl\vert\,\Delta (\boldsymbol \eta* \boldsymbol \lambda * \boldsymbol \alpha  ;q,a)\, \Bigr\vert \ll  x \mathcal L^{-A },
\end{equation}
as soon as $Q$ satisfies
 $$
Q\mathcal L^{A_0} < KL,\ MK^4Q <  x^{2-\eps}\text{ and } MK^2Q^2 < x^{2-\eps}.
\leqno (S4)
$$
 The constant implied in the $\ll$--symbol in \eqref{22fev10} depends at most on $\eps$, $a$, $B$, $C$.
 \end{proposition} 
Finally we recall a consequence of bounds of exponential sums (coming either from Weil's or Deligne's work) and of the fundamental lemma in sieve theory. We have (see \cite[Lemma 2*]{BFI2})
\begin{proposition} \label{deligne}Let $\eps >0$. Let $K$, $L$ and $M\geqslant 1$ and $x\geqslant KLM$, such that either $K=1$ or $K\geqslant x^{12\eps}$ and similarly either $L=1$ or $L\geqslant x^{12\eps} $ and either $M=1$ or $M\geqslant x^{12\eps}$.
Let $\mathfrak K$, $\mathfrak L$ and $\mathfrak M$ be three intervals respectively included in $[K, 2K[$, $[L, 2L[$ and $[M, 2M[$. Then there exists an absolute positive  constant $\delta$ such that 
$$
\Delta \bigl(\mathfrak z ({\mathbf 1}_{\mathfrak K}* {\mathbf 1}_{\mathfrak L}*{\mathbf 1}_{\mathfrak M} );q,a)   \ll {x\over \varphi (q)} \exp\Bigl( -\eps{\log x\over \log z}\Bigr)
$$
uniformly  for $(q,a)=1$ and $q\leqslant x^{{1\over 2}+\delta}$. The constant implied in the $\ll$--symbol depends at most on $\eps$.
\end{proposition}

\section{Proof of Theorem \ref{extension}}\label{goodday}
We arrive now at the central part  of our work. Of course, our proof   
highly imitates the proof given in \cite{BFI2}. The combinatorics is heavy and we were unable to find shortcuts to simplify the technique of \cite{BFI2}.

 \subsection{Notations and first reductions of the proof of Theorem \ref{extension}}\label{notation} As in \cite[\S 13]{BFI2},  the notation 
 $$
 \sum_n{}^*
 $$
 means that we are summing over integers $n$, with $\mathfrak z (n) =1,$
 and 
 $z$ now  has the value
 \begin{equation}\label{sam1}
 z:=\exp\Bigl({\log x\over (\log \log x)^2}\Bigr).
\end{equation}
 To prove Theorem \ref{extension}, we consider
 $$
 S(x, Q,P_1,P_2) :=  \sum_{q\sim  Q\atop (q,a)=1} \Bigl\vert 
\underset{P_1< p\leqslant P_2,\ pm\leqslant x\atop pm\equiv a \bmod q}{ \sum\ \ \  \sum} \log p  -{1\over \varphi (q)}
\underset{P_1< p\leqslant P_2,\ pm\leqslant x\atop (pm, q)=1}{ \sum\ \ \  \sum} \log p\,
 \Bigr\vert\,.
 $$
   Since $a$ is supposed to be fixed, we forget the dependency on $a$. We want to prove 
   the inequality
 \begin{equation}\label{16jan1}
 S(x,  Q,P_1,P_2)\ll  
  x\cdot {(\log y)^2\over \log x} \cdot  (\log \log x)^{B_2},
 \end{equation}
under the   conditions of Theorem \ref{extension}. 
For the rest of the proof,  we suppose the inequalities $B_2 \geqslant 2$ and  $Q^2\leqslant xy$ with
 \begin{equation}\label{16jan2}
\mathcal L^A \leqslant  y\leqslant \exp\Bigl({\log x\over  \log \log x}\Bigr):= y_0,
 \end{equation}
since $y\geqslant y_{0}$, 
 \eqref{16jan1} is trivial, by \eqref{20jan1}.  In \eqref{16jan2}, $A$ is a constant the definition of which will be given in \S \ref{interior}  when applying the results of \S \ref{equidistribution} -- \ref{rabbit}.  
 We shall also assume that
 $$
 Q\geqslant  x^{{1\over 2}-\eps},
 $$
 otherwise,  \eqref{16jan1} is a direct consequence of Proposition \ref{largesieve}. Of course Theorem  \ref{extension} is trivial also  when $P_1$ is too large ($P_1 >x$, since the sum is empty). Using the classical formulas
$$
\sum_{m\leqslant {x\over p}\atop m\equiv b\bmod q} 1 ={x\over p q} +O(1),
$$
and  
  \begin{equation}\label{31jan10}
 \sum_{m\leqslant {x\over p} \atop (m,q)=1} 1 ={\varphi (q) \over q}\cdot {x\over p} +O(\tau (q)),
\end{equation}
 we deduce the inequality
 $$
 S(x, Q,P_1,P_2) \ll \sum_{q\sim Q} \  \tau (q)   \sum_{P_1< p\leqslant P_2} \log p
 \ll    P_2\ Q \LL.
 $$
  This implies  that \eqref{16jan1} is trivially true when $P_2\leqslant x^{1\over 2} y_0^{-1} $ since     \eqref{16jan2} is  satisfied by hypothesis.  So we can restrict to the case $P_2 > x^{1\over 2} y_0^{-1} $.  Furthermore, if $P_1 < x^{1\over 2} y_0^{-1} < P_2$, we split the interval $]P_1,P_2]$ into the two intervals $]P_1, x^{1\over 2} y_0^{-1}] $ and  $]x^{1\over 2}y_0^{-1}, P_2]$. The contribution of the second case is trivially solved by the above remark. Hence, we can even  restrict ourselves to the case
 \begin{equation}\label{hurry}
 P_2 \geqslant P_1\geqslant x^{1\over 2} y_0^{-1}.
 \end{equation}
 Replacing the factor $\log p$ by the van Mangoldt function $\Lambda (n)$ (with an acceptable error) and applying a dyadic dissection we are led to introduce the modified sum
\begin{equation}\label{20jan2}
   {\mathcal E}(x,  Q,P_1,P_2) :=  \sum_{q\sim  Q\atop (q,a)=1} \Bigl\vert 
\underset{P_1< n\leqslant P_2,\  x/2<nt\leqslant x\atop nt\equiv a \bmod q}{ \sum{}^*\ \ \  \sum} \Lambda (n)  -{1\over \varphi (q)}
\underset{P_1< n\leqslant P_2,\ x/2<nt\leqslant x\atop (nt, q)=1}{ \sum{}^*\ \ \  \sum} \Lambda (n)\,
 \Bigr\vert.
\end{equation}
Note that the variable $m$ is now called $t$ to prepare the applications of some lemmas of  \S \ref{divisorfunctions} and that, in the summation of \eqref{20jan2}, we never $nt=a$ for sufficiently large $x$; this is a consequence of \eqref{hurry} and the fact that $a$ is fixed.  The sum
${\mathcal E}(x,  Q,P_1,P_2)$   is the analogue of $\mathcal E (x,Q)$ introduced in \cite[p.388]{BFI2}. Gathering the above remarls, 
  the proof 
 of  \eqref{16jan1} is equivalent to the proof of  the inequality
 \begin{equation}\label{16jan4}
 \mathcal E (x,  Q,P_1,P_2)\ll       x\cdot {(\log y)^2\over \log x} \cdot  (\log \log x)^{B_2},
 \end{equation}
 under the restrictions \eqref{16jan2} and \eqref{hurry} and the condition $Q^2\leqslant xy$.

\subsection{Preparation of the variable $t$}  In \eqref{20jan2}, the variable $t$ may have prime divisors less than the parameter  $z$ defined in \eqref{sam1}. This implies that  we cannot directly use the characteristic function of the set $\{t\}$ to build (by convolution  with other variables)  a sequence $(\beta_n)$ in order to apply one the propositions  of the \S\ref{equidistribution} -- \ref{rabbit}, since   the assumption $(A_3(x))$ may be  not   satisfied. To circumvent this almost primality condition, we can proceed as follows. We factorize each $t$ as $t=t^\dag \cdot t^\ddag$, where
$$
t^\dag =\prod_{p^\nu\, \Vert \, t\atop p<z} p^\nu.
$$
This factorization is unique. Usually, $t^\dag$ is small  compared with $t$ since we have (see \cite[Theorem 07 p.4]{HaTe})
\begin{lemma}\label{19:12}
There exists an absolute positive $c_0$, such that,  uniformly for $     v\geqslant u\geqslant 2 $ and $ x\geqslant 2$  we have the inequality
$$
\Theta (x; u,v):=\bigl\vert \{ n\leqslant x\ ;\, \prod_{p^\nu\, \Vert \, n\atop p\leqslant u}p^\nu \geqslant v\}
\bigr\vert \ll  x \exp\Bigl( -c_0\, {\log v\over \log u}\Bigr).
$$
\end{lemma}
Let  $W_0$ be a number which satisfies
\begin{equation}
\label{defT0}
W_0\sim \exp\Bigl({\log x\over \sqrt{ \log \log x}}\Bigr).
\end{equation}
By Lemma \ref{19:12}, the contribution of the triples  $(q,n,t)$ such that $t^\dag > W_0$ to the right part of \eqref{20jan2} is
\begin{equation}\label{station}
\ll \mathcal L \Bigl(\underset{  x/2 < nt\leqslant x \atop t^\dag >W_0}{\sum{}^* \ \sum} \tau (nt-a) +     \underset{   \ x/2< nt\leqslant x \atop t^\dag >W_0}{\sum{}^* \ \sum} 1\Bigr).
\end{equation}
We  write $m=nt$, to see that the expression \eqref{station} is  
$$
\ll \mathcal L \, \sum_{x/2 <m\leqslant  x  \atop m^\dag >W_0} \tau (m) \tau (m-a)
\ll \mathcal L  \cdot \Theta^{1\over 3} (x;  z, W_0) \cdot \Bigl( \sum_{m\leqslant  x} \tau^3 (m)\Bigr)^{1\over 3} \cdot\Bigl( \sum_{x/2 < m\leqslant x } \tau^3 (m-a)\Bigr)^{1\over 3},
$$
by H\" older's inequality.  Appealing to Lemmas \ref{linnik1} \& \ref{19:12}, we see that the above term is
$$
\ll x\cdot \mathcal L^6\cdot \exp\Bigl( -{c_0\over 3} (\log \log x)^{3\over 2}\Bigr) \ll x\mathcal L^{-C}\text{ for all } C,
$$
which is acceptable in view of \eqref{16jan4}, that we want to prove.
These considerations allow to  replace $t$ by $t=uw$ where $u$ and $w$ have all their prime factors either greater than $z$ or smaller than $z$. In others words we are reduced to  prove the inequality \eqref{16jan4} for  the sum $\tilde {\mathcal E} (x,Q,P_1,P_2)$ defined by
\begin{equation}\label{deftildeE}
\tilde{\mathcal E}(x, Q,P_1,P_2) :=  \sum_{q\sim  Q\atop (q,a)=1} \Bigl\vert 
\underset{P_1< n\leqslant P_2,\ x/2< nuw\leqslant x\atop nuw\equiv a \bmod q}{ \sum{}^*\  \sum{}^*\ \  \sum{}^\dag} \Lambda (n)  -{1\over \varphi (q)}
\underset{P_1< n\leqslant P_2,\ x/2< nuw\leqslant x\atop (nuw, q)=1}{ \sum{}^*\ \sum{}^*\ \  \sum{}^\dag} \Lambda (n)\,
 \Bigr\vert,
\end{equation}
where the ${}^\dag$--symbol means that $w$ has all its prime factors less than $z$ and where $w$ is a small variable, which means that it satisfies
\begin{equation}\label{dim1}
w\leqslant W_0.
\end{equation}
 \subsection{Application of a combinatorial identity} To transform  
 the function $\Lambda$ into bilinear forms, we appeal to the identity of Heath--Brown (see \cite [Prop. 13.3]{IwKo} for instance)
 \begin{lemma}\label{HeathB}
 Let $J\geqslant1$ and $n <2x$. We then have the equality
 $$
 \Lambda (n)= \sum_{j=1}^J (-1)^j \biggl({ J\atop j}\biggr) \sum_{m_1,\dots,m_j  \leqslant x^{1/J}}\mu (m_1)\cdots \mu (m_j)\sum_{m_1\dots m_jn_1\dots n_j=n}\log n_1.
 $$
 \end{lemma}
 We apply this lemma  to $\Lambda (n)$ inside  \eqref{deftildeE} with $J=7$.   It gives
 the inequality
 \begin{equation}\label{17jan1}
 \tilde {\mathcal E} (x,  Q,P_1,P_2) \leqslant \biggl( {7\atop 4} \biggr) \sum_{j=1}^7  \, \tilde{ \mathcal E}_j (x,  Q,P_1,P_2),
 \end{equation}
 with
 \begin{align}\label{17jan2}
\tilde{ \mathcal E}_j (x,  Q,P_1,P_2):= \sum_{q\sim  Q\atop (q,a)=1} \Bigl\vert &\underset{m_1\cdots m_jn_1\cdots n_j uw\equiv a \bmod q}{\sum{}^* \cdots \sum{}^*\sum{}^\dag}\mu (m_1)\cdots \mu (m_j) \log n_1\\
 & -
 {1\over \varphi (q)}  \underset{(m_1\cdots m_jn_1\cdots n_j uw , q)=1}{\sum{}^* \cdots \sum{}^*\sum{}^\dag}\mu (m_1)\cdots \mu (m_j) \log n_1
 \Bigr\vert ,\nonumber
 \end{align}
 where the variables of summation satisfy the inequalities 
\begin{align}\label{17jan4}
x/2<m_1\cdots m_jn_1\cdots n_j uw\leqslant x,\ & P_1 <m_1\cdots m_jn_1\cdots n_j \leqslant P_2\\
&\text{ and } m_1,\dots , m_j \leqslant D,\nonumber
 \end{align}
 where $D$ is any number $\geqslant x^{1\over 7}$.
 and where  $w$ satisfies  \eqref{dim1}.
 \subsection{Dissection of the set of summation.} In \eqref{17jan4}, the variables $m_i$, $n_i$, $u$ and $w$ have to satisfy several multiplicative inequalities. To make these variables independent we process as usual in such problems, see \cite[p.388]{BFI2} for instance. We define a parameter $\delta$ satisfying $x^{-\eps} < \delta <1$, and introduce the notation
 $$
 g\simeq G $$
 to mean that the  integer variable   $g$ satisfies $G\leqslant g< (1+\delta) G$.  Let
 $$
 \mathcal D: =\bigl\{ (1+\delta)^\nu\,;\ \nu=0,\,1, \, 2,\dots \bigr\}.
 $$
 To transform \eqref{17jan4}, we precise \eqref{defT0} and the choice of $D$ by imposing
 $$
 D,\, W_0\in \mathcal D  \text{ and }  D\simeq x^{1\over 7}.   
 $$
  The conditions  \eqref{17jan4} are equivalent to
 \begin{equation}\label{fire}
 m_i\simeq M_i,\ n_i\simeq N_i \ (1\leqslant i \leqslant j), \ u\simeq U
 \text{ and } w\simeq W,
\end{equation}
 for some numbers $M_1$,..., $M_j$, $N_1$,..., $N_j$, $U$ and $W$ from $\mathcal D$ satisfying 
 \begin{align}\label{17jan5}
 \begin{cases}x/2<M_1\cdots M_jN_1\cdots N_j UW\leqslant x,\\ 
  P_1 <M_1\cdots M_jN_1\cdots N_j \leqslant P_2,\\
   M_1,\dots , M_j \leqslant D/(1+\delta),\\
   W\leqslant  W_0/(1+\delta),
 \end{cases}
 \end{align}
 unless the $2j+2$--uple   $(m_1,\cdots ,m_j,n_1,\cdots ,n_j ,u,w)$ is too near from  some edge  of the dissection, that means satisfies  at least one of the following four  conditions
\begin{align}\label{21jan1}
\begin{cases}
  x< m_1\cdots m_jn_1\cdots n_j uw\leqslant x(1+\delta)^{2j+2},\\
  x/2< m_1\cdots m_jn_1\cdots n_j uw\leqslant (x/2)(1+\delta)^{2j+2},\\
  P_1 <m_1\cdots m_jn_1\cdots n_j\leqslant P_1 (1+\delta)^{2j}\   \text{ and } \  m_1\cdots m_jn_1\cdots n_j uw\leqslant  x,  \\
   P_2 <m_1\cdots m_jn_1\cdots n_j\leqslant P_2 (1+\delta)^{2j}\  \text{ and } \  m_1\cdots m_jn_1\cdots n_j uw\leqslant  x.  \\
\end{cases}
\end{align}
Write $r:=m_1\cdots m_jn_1\cdots n_j$ and $t:=uw$ (note that each $t$ can be written in an unique way in that form, since $u$ (resp. $w$) has all its prime factors greater (resp. less) than $z$). It is easy to see that the contribution (denoted by $\mathcal C_{1,j}$)
to $\tilde{\mathcal E}_j (x,Q,P_1,P_2)$ of  the $(2j+2)$--uples
$(m_1,\cdots ,m_j,n_1,\cdots ,n_j ,u,w)$ which satisfy  at least one of the four conditions of \eqref{21jan1} is 
\begin{align}\label{21jan2}
\mathcal C_{1,j}&\leqslant \mathcal L \,\Big\{\sum_{x <rt\leqslant x(1+\delta)^{2j+2}\text{ or }\atop x/2 <rt\leqslant (x/2)(1+\delta)^{2j+2} }  +  
  \sum_{ P_1 <r\leqslant P_1 (1+\delta)^{2j}\atop x/2<rt\leqslant x} +\sum_{ P_2 <r\leqslant P_2 (1+\delta)^{2j}\atop x/2 <rt\leqslant x} \Bigr\}\, \tau (rt-a) \tau_{2j} (r)\\
&+\mathcal L\,\sum_{q\sim Q}{1\over \varphi (q)}  \Big\{\sum_{x <rt\leqslant x(1+\delta)^{2j+2} \text{ or }\atop x/2 <rt\leqslant (x/2)(1+\delta)^{2j+2}}  +  
  \sum_{ P_1 <r\leqslant P_1 (1+\delta)^{2j}\atop x/2< rt\leqslant x} +\sum_{ P_2 <r\leqslant P_2 (1+\delta)^{2j}\atop x/2< rt\leqslant x} \Bigr\} \,  \tau_{2j} (r).\nonumber
\end{align}
   Write $s:= rt$. Then we see that in the six inner  summations of  \eqref{21jan2},  $s$ belongs to some set $\mathcal S\subset [x/2, 2x]$, which satisfies 
\begin{equation}\label{21jan3}
\vert \mathcal S\vert  \ll x \bigl[ (1+\delta)^{2j+2} -1\bigr] +x \log (1+\delta)^{2j}\ll \delta x.
\end{equation}
Hence \eqref{21jan2} implies that
\begin{align}
 \mathcal C_{1,j} &\ll \mathcal L\sum_{s\in \mathcal S} \tau (s-a) \tau_{2j } (s) + \mathcal L   \sum_{s\in \mathcal S} \tau_{2j} (s)\nonumber \\
 &\ll \mathcal L \cdot \vert \mathcal S \vert^{1\over 3} \cdot \Bigl\{ \sum_{x/2 \leqslant s\leqslant 2x } \tau^3 (s-a)\Bigr\}^{1\over 3} \cdot  \Bigl\{ \sum_{x/2 \leqslant s\leqslant 2x } \tau_{2j}^3 (s)\Bigr\}^{1\over 3}\nonumber  \\
 & \ll \delta^{1\over 3}\,  x\,\mathcal L^{B_3},\label{21jan200}
\end{align}
for some absolute positive $B_3$
 by appealing to \eqref{21jan3} and to Lemma \ref{linnik1}.
 
 Finally, remark that if $n_1\simeq N_1$, the function $n_1 \mapsto \log n_1$ is almost constant, more precisely, we have
\begin{equation}\label{logLog}
 \log n_1 =\log N_1 + O(\delta).
 \end{equation} Hence, replacing $\log n_1$ by $\log N_1$ in the expression \break $\tilde{\mathcal E}_j (x,Q,P_1,P_2)$, we create a global  error $\mathcal C_{2,j}$ that we majorize by 
 \begin{align}
 \mathcal C_{2,j}& \ll \delta \sum_{q\sim Q\atop (q,a)=1} \Bigl( \sum_{x/2 < n\leqslant x \atop n\equiv a \bmod q} \tau_{2j+1} (n) +
 {1\over \varphi (q)} \sum_{x/2 < n\leqslant x} \tau_{2j+1} (n)\Bigr)\nonumber\\
 &\ll \delta \sum_{x/2<n\leqslant x} \tau (n-a)\tau_{2j+1} (n)  +\delta \, x\,   \mathcal L^{2j} \nonumber\\
 &\ll \delta \, x\, \mathcal L^{B_4},\label{21jan100}
 \end{align}
 for some positive  absolute $B_4$. Here also, we used Cauchy--Schwarz inequality and Lemma \ref{linnik1} to prove \eqref{21jan100}.
 The constant $\delta$ is at our disposal, so we fix
 \begin{equation}\label{21jan10}
 \delta := \mathcal L^{-3(2+B_3+B_4)}. 
 \end{equation}
 Hence, by  \eqref{21jan200} \& \eqref{21jan100},  the error terms $\mathcal C_{1,j}$ and $\mathcal C_{2,j}$ both satisfy
 $$
 \mathcal C_{1,j},\ \mathcal C_{2,j} \ll x \mathcal L^{-2},
 $$
 which, in view of \eqref{17jan1},  is acceptable (compare with \eqref{16jan4}).

 It remains to prove, for $1\leqslant  j\leqslant  7$,  the inequality 
\begin{align}\label{15:18}
 \mathcal F_j (x,  Q,P_1,P_2)\ll 
  x\, \Bigl( {\log y\over \log x}\Bigr)^2\, (\log \log x)^{B_2},  
 \end{align}
 $\bullet$ where 
\begin{equation}\label{13:17}
  \mathcal F_j (x,  Q,P_1,P_2) :=\sum_{\mathcal M,\, \mathcal N, \,U,\, W}
  \mathcal F_j(\mathcal M, \mathcal N, U,W,   Q),
  \end{equation}
$\bullet$ with $\mathcal M= (M_1,\dots, M_j)$, and $\mathcal N = (N_1,\dots, N_j)$, 

\noindent $\bullet$ where the numbers $M_i$, $N_i$ ($1\leqslant i\leqslant  j$), $U$ and $W$ are taken in the set $\mathcal D$  and satisfy \eqref{17jan5},

\noindent $\bullet$  and where 
 $$
 \mathcal F_j(\mathcal M, \mathcal N, U,W,  Q): = \sum_{q\sim  Q\atop  (q,a)=1  } \bigl\vert
 \mathcal F_j (\mathcal M ,\mathcal N ,U,W; q,a)\bigr\vert,
 $$
$ \bullet$ with
\begin{align}\label{defF}
 \mathcal F_j (\mathcal M ,\mathcal N ,U,W; q,a):=&
 \underset{m_i\simeq M_i,\dots, n_i\simeq N_i \,  (1\leqslant i\leqslant j),\  u\simeq U,\, w\simeq W\atop m_1\cdots m_j n_1\cdots n_j uw\equiv a \bmod q}{\sum{}^*\cdots \sum{}^*\ \sum{}^*\ \sum{}^\dag} \mu (m_1)\cdots \mu (m_j)\\
& -{1\over \varphi (q)}
  \underset{m_i\simeq M_i,\dots, n_i\simeq N_i\, (1\leqslant i\leqslant j),\  u\simeq U,\, w\simeq W\atop ( m_1\cdots m_j n_1\cdots n_j uw,q)=1}{\sum{}^*\cdots \sum{}^*\  \sum{}^*\sum{}^\dag}\mu (m_1)\cdots \mu (m_j). \nonumber
\end{align}
Remark that in the conditions of summation \eqref{defF}, the variable $w$ has all its prime factors smaller than $z$, this fact creates extra difficulty compared with \cite{BFI2}. The factor $\log n_1$ present in \eqref{17jan2} has  now disappeared thanks to \eqref{logLog} with the price of an extra log--factor in right part of  \eqref{15:18} that we want to prove.
\subsection{The boundary configuration}\label{boundaryconf} We follow the technique of \cite[\S\, 13]{BFI2}. We fix 
\begin{equation}\label{defDelta}
  \Delta := y^6, 
\end{equation}
hence $\Delta$ satisfies $1\leqslant \Delta < x^\eps$ by \eqref{16jan2}.
 Given an integer $r\geqslant 1$ and  a finite sequence  $(D_1,\dots, D_r)$ of  elements of $\mathcal D$ such that
$$
D_1\geqslant \cdots \geqslant D_r\geqslant D \text { and } D_1\cdots D_r <x,
$$
we say that it is a {\it boundary configuration} if one of the following holds
$$
r=4,\quad D_1 \leqslant \Delta D_2\quad  \text{and} \quad D_3 \leqslant  \Delta D_4,\leqno (B_4)
$$
$$
r=5, \quad D_3 \leqslant \Delta D_5, \leqno(B_5)
$$
$$
r=6, \quad D_4 \leqslant \Delta D_6. \leqno (B_6)
$$ 
Let $1\leqslant j\leqslant 7$. A sequence $(\mathcal M, \mathcal N, U,W)=(M_1,\dots, M_j, N_1,\dots, N_j, U,W)$ is said to be {\it exceptional  of type} $B_r$ (with $4\leqslant r\leqslant 6$) if it can be partioned into subsets whose products form a boundary configuration 
$(D_1,\dots, D_r)$ satisfying $(B_r)$. A sequence is {\it good } if it is not exceptional of any type. By \eqref{13:17} we get the inequality
\begin{align}\label{13:19}
& \mathcal F_j (x, Q,P_1,P_2)\\
 & =\sum_{r=4}^6\sum_{(\mathcal M, \mathcal N,\, U,\,W)\atop \text{exceptional of type } B_r}
  \mathcal F_j(\mathcal M, \mathcal N, \, U,\, W,  Q) + \sum_{(\mathcal M, \mathcal N, U,W)\atop \text{ good}  }
  \mathcal F_j(\mathcal M, \mathcal N, \, U,\, W,  Q), \nonumber
  \end{align}
  where $(\mathcal M, \mathcal N, U,W)$ satisfy \eqref{17jan5}.
  We now treat the contribution of  the $(\mathcal M, \mathcal N, U,W)$
  of type $(B_r)$ to the right part of  \eqref{13:19}. We shall concentrate on $r=4$ (necessarily, we have $j\geqslant 2$). So we can partition the variables of summation $m_1,\dots,m_j, n_1,\dots, n_j, u,w$ into four products $d_1$, $d_2$, $d_3$, $d_4$ which 
  necessarily satisfy
  $$
 n:= m_1\dots m_j n_1\dots n_j uw=d_1d_2d_3d_4 \leqslant 2x,
  $$
  and
  $$
  d_1\geqslant d_2 \geqslant d_3\geqslant d_4 \geqslant D\simeq x^{1\over 7}, \ d_1\leqslant 2\Delta d_2,\ d_3\leqslant 2 \Delta d_4.
  $$
 The number of representations of $n$ in the form
 $n=m_1\dots m_jn_1\dots n_juw$ is bounded by $\tau_{2j+1} (n)$. The study must be seperated in four cases according to the index $i$ ($1\leqslant i\leqslant 4$) such that the variable $w$ appears in the 
 constitution of the variable $ d_i$ in the above partition. We   shall only write the case where $i=1$ with details (the other cases 
$i=2$, $i=3$ or $i=4$ can be similarly handled by using variants  \eqref{13:50} quoted at the end  of Lemma \ref{22jan1}).  Note that, when $i=1$,  the variables $d'_1:=d_1/w$, $d_2$, $d_3$ and $d_4$ all have their prime factors greater than  $z$, but such a condition does not apply to $w$. 
We fix     $y'=2\Delta$ and we write 
\begin{align}\label{sam2}
&\sum_{\substack{(\mathcal M,\, \mathcal N,\, U,\, W) \text { of type } B_4\\
 w\text{ appears  in } d_1}} \mathcal F_j (\mathcal M, \mathcal N, U,W;q,a)\nonumber\\
 &\leqslant 
  \underset
  {\substack{wd'_1d_2d_3d_4\leqslant 2x \\
  D\leqslant d_4\leqslant d_3 \leqslant d_2 \leqslant w d'_1\\
   d_3\leqslant y' d_4,\, wd'_1\leqslant y' d_2\\
 wd'_1d_2d_3d_4\equiv a \bmod q }}
 {\sum\ \sum{}^*\ \sum{}^*\  \sum{}^*\ \sum{}^*}
  \tau_{2j+1} (d_4)\tau_{2j+1} ( d'_1d_2d_3)\nonumber \\
  &+{1\over \varphi ( q)}
  \underset
  {\substack{wd'_1d_2d_3d_4\leqslant 2x \\
  D\leqslant d_4\leqslant d_3 \leqslant d_2 \leqslant w d'_1\\
   d_3\leqslant y' d_4,\, wd'_1\leqslant y' d_2 
  }}
 {\sum\ \sum{}^*\ \sum{}^*\  \sum{}^*\ \sum{}^*}
  \tau_{2j+1} (d_4)\tau_{2j+1} ( d'_1d_2d_3)\nonumber\\
&  :=F_{4,1} (q,a)+{1\over \varphi (q)} F_{4,1},
\end{align}
 by definition. The index $1$ is to remember that the variable $w$ is glued to $d'_1$.
   Formula \eqref{11:09} of Lemma \ref{22jan1}  gives the bound 
\begin{equation}\label{sam4}
 F_{4,1} \ll (\log \log x)\cdot {x\over \log 2D} \cdot {(\log 2y')^2\over \log 2x}\cdot 
 \Bigl({\log x\over \log z}\Bigr)^{8j+4}\ll x\cdot\Bigl( {\log  y\over \log  x}\Bigr)^2 (\log\log x)^{B_5},
\end{equation}
 by  \eqref{sam1}, \eqref{defDelta} and the inequality $\tau_{2j+1} (d'_1d_2d_3)\leqslant \tau_{2j+1} \tau_{2j+1} (d_2)\tau_{2j+1} (d_3)$. Here $B_5$ is an absolute constant. Summing over $q\sim  Q,\ (q,a)=1$, we recognize the second term in the inequality  \eqref{15:18}. By \eqref{sam2}, we must give an upperbound for 
 \begin{equation}\label{sam3} 
0\leqslant \sum_{q\sim Q\atop (q,a)=1} F_{4,1}(q,a)\leqslant 
 \sum_{q\sim Q\atop (q,a)=1} \Bigl( 
F_{4,1} (q,a)   -{1\over \varphi (q)} 
  \sum_{b\bmod q\atop (b,q)=1} F_{4,1}(q,b)   \Bigr) +\sum_{q  \sim Q\atop (q,a)=1}{1\over \varphi (q)} F_{4,1}.
\end{equation}
The second term  on the right part of \eqref{sam3} is treated as \eqref{sam4}. For the first one, we appeal to  Proposition \ref{SAM1} with the following choice of the parameters 
$$
\xi (\ell, m) =\underset{wd'_1d_2=\ell, d_3 =m\atop d_3\leqslant d_2\leqslant wd'_1\leqslant y' d_2} {\sum \ \sum{}^*\  \sum{}^*\ \sum{}^*}
    \tau_{2j+1}  ( d'_1d_2d_3)
  $$
  and
 \begin{equation}\label{18:19}
  \beta_n =\underset{ D \leqslant d_4=n< (3x)^{1\over 4}}{\sum{}^*}\tau_{2j+1} (d_4),
 \end{equation}
  $y_1=1$ and $y_2 =y'$. By \eqref{sam3}, we deduce 
  \begin{equation}\label{sam5}
  \sum_{q\sim Q\atop (q,a)=1}
  F_{4,1} (q,a)\ll x \mathcal L^{-2}  
 +   x\cdot \Bigl({ \log y \over \log x}\Bigr)^2 \cdot  (\log \log x)^{B_5},
 \end{equation}
 The same procedure applies when $w$ participates to $d_2$ or $d_3$. We then apply \eqref{13:50} of Lemma \ref{22jan1}. However, when $w$ participates to $d_4$, the choice of the function $\beta_n$ given in \eqref{18:19} is not correct, since the assumption $(A_3(x))$ is not satisfied to apply Proposition \ref{SAM1}. We then choose $n=d_3$, $m=d_4$, $y_1=1/y'$ and $y_2=1$ to apply this Proposition. Note the inequality $x^{1\over 7} \leqslant n\leqslant 3 x^{2\over 7}$ in all the cases.  
 
 It remains to deal with the contribution of the cases $r=5$ and $r=6$ to the right part of \eqref{13:19}. The study is the same as for $r=4$, but Lemma \ref{15:10} will now replace Lemma \ref{22jan1}.

 In conclusion, we proved that the contribution of the boundary is acceptable, which means  that it satisfies the inequality
 \begin{equation}\label{16:30}
 \sum_{r=4}^6\sum_{(\mathcal M, \mathcal N,\, U,\,W)\atop \text{exceptional of type } B_r}
  \mathcal F_j(\mathcal M, \mathcal N, \, U,\, W,   Q) \ll 
  x\cdot \Bigl({ \log y \over \log x}\Bigr)^2 \cdot  (\log \log x)^{B_5},
 \end{equation}
 for $1\leqslant j\leqslant 7$, with an absolute $B_5$.

\subsection{The interior}\label{interior} This subsection has to be compared with \cite[\S 15]{BFI2}.   For any $1\leqslant j\leqslant 7$, the number of good subsequences  $(\mathcal M, \mathcal N, U,W)$   is $O( (\delta^{-1} \mathcal L)^{16})=O(\mathcal L^{B_6})$, by the choice \eqref{21jan10}, for some  absolute constant $B_6$. Hence, by \eqref{13:19} and \eqref{16:30},  the inequality \eqref{15:18} will be proved as soon as, for each $j$, with $1\leqslant j\leqslant 7$, for each good sequence  $(\mathcal M, \mathcal N, U,W)=(M_{1}, \dots,M_{j}, N_{1},\dots, N_{j}, U,W)$ satisfying  
\begin{align}\label{3fev1}
 \begin{cases}x/2 <M_1\cdots M_jN_1\cdots N_j UW\leqslant x,\\ 
   M_1,\dots , M_j \leqslant D/(1+\delta),\\
   W\leqslant  W_0/(1+\delta),
 \end{cases}
 \end{align}
  we have the inequality
\begin{equation}\label{11:22}
\mathcal F_j (\mathcal M, \mathcal N, U,W,Q)\ll x\mathcal L^{-B_7},
\end{equation}
where $B_7$ is some absolute constant, for instance  $B_7=1+B_6$. The conditions \eqref{fire},  \eqref{defF} and  \eqref{3fev1}   show that the variable $u$ behaves like any variable  $n_i$. So it is natural to replace the name of the variable $u$  
by $n_{j+1}$ and $U$ by  $N_{j+1}$. We put $\tilde {\mathcal N} =(\mathcal N, U)$ and     $(\mathcal M, \mathcal N, U,W)$ becomes $(\mathcal M, \tilde{\mathcal N}, W)$. If   $(\mathcal M, \tilde{\mathcal N}, W)$ is  a good sequence, by symmetry, 
we may assume that
$$
N_1\geqslant N_2\geqslant \cdots\geqslant N_{j+1},
$$
and we denote by $N_{s+1}, \cdots, N_{j+1}$ those $N_i$ (possibly none) which are $\leqslant x^{{1\over 6}-\eps}$.

As in \cite{BFI2}, the rest of the proof consists in playing with the orders of magnitude of 
$M_i$($1\leqslant i \leqslant j$), $N_i$ ($1\leqslant i\leqslant j+1$)  and $W$ to apply one of the propositions of \S \ref{equidistribution},
\S \ref{convolof3}, \S \ref{convolof3fouvry} or \S \ref{rabbit}. However, $W$ is always small ($\leqslant  W_0\ll x^\eps$)  and the characteristic function of the integers $w\simeq W$ in question never satisfies $(A_3(x))$, we shall also play with  this small variable. This creates unexpected problems.
\vskip.2cm
\noindent $\bullet$ {\it Case 1.} $M_1\dots M_jN_{s+1}\dots N_{j+1}\geqslant x^\eps$.
  Then, there exists a partial product, say $V$,  of  some of these $M_i$ and $N_i$,   which satisfies $x^\eps
  \leqslant V \leqslant x^{{1\over 6}-\eps}.$ This is a consequence of the inequalities   $M_i\leqslant D$  and $x^{{1\over 6}-\eps} \geqslant N_{s+1}\geqslant \cdots \geqslant N_{j+1}$.
We apply Proposition \ref{17:09}  with $N=V$, and $M$ such that $MN= M_1\dots M_jN_1\dots N_{j+1}W$ (which is $\sim x/2$ by 
\eqref{3fev1}). We also define $\boldsymbol \beta$ as the convolution
of  the functions $(\mathfrak z \mu)$ and $(\mathfrak z)$, with respective support $\simeq M_i$ and $\simeq N_i$, where the $M_i$ and $N_i$ are those parameters  which participate to the partial product $V$. Hence, in that case \eqref{11:22} is proved.

So we are left with the case
\begin{equation}\label{assumption1}
V:=M_1\dots M_jN_{s+1}\dots N_{j+1}\leqslant x^\eps \text{ and }N_1\geqslant \cdots N_s \geqslant x^{{1\over 6}-\eps}.
\end{equation}
By \eqref{3fev1}, we necessarily have $1\leqslant s \leqslant 6$. The value of $s$ is an important parameter of our discussion.

\vskip .2cm
\noindent $\bullet$ {\it Case 2.} $s=6$. We partition $(\mathcal M, \tilde {\mathcal N}, W)$ into 
$$ (N_1VW)\geqslant N_2\geqslant N_3\geqslant N_4\geqslant N_5 \geqslant N_6\ (\geqslant x^{{1\over 6}-\eps}),$$
and since $(\mathcal M, \tilde{\mathcal N}, W)$ is good, we necessarily have $N_4>\Delta N_6$. We want to apply  Proposition \ref{fouvrytrilinear}, with $L:= (N_1VW)N_2$, $M:=N_4N_5N_6$ and $N:= N_3.$  We easily check the conditions $(S3)$, since, by \eqref{assumption1}, we have $L=x^{{1\over 3} +O(\eps)}$, $M=x^{{1\over 2}+O(\eps)}$ and $N=x^{{1\over 6}+O(\eps)}$. 

It remains to check the inequality $Q\leqslant (LN)\mathcal L^{-A_0}$ to deduce that, in that case, \eqref{11:22} is satisfied. We write
$$
x/2 \leqslant VW N_1N_2N_3N_4N_5N_6\leqslant 
VW N_1N_2N_3N_4^2N_5\Delta^{-1}\leqslant {VW\over \Delta} (N_1N_2N_3)^2,
$$
which implies $x/2 \leqslant (LN)^2/\Delta$, 
  from which we deduce 
  $$
  LN \geqslant (\Delta x/2)^{1\over 2}= (x/2)^{1\over 2} y^3\geqslant   Q y\geqslant Q \mathcal L^{A }, $$
by   \eqref{16jan2} and  \eqref{defDelta}. Finally, if we choose 
$A$ larger than $A_0 (=A_0 (B_7))$ where $A_0$ is  defined in Proposition 
\ref{fouvrytrilinear}, we have $Q\leqslant (LN)\mathcal L^{-A_0}$. Proposition \ref{fouvrytrilinear} is applicable and   \eqref{11:22} is proved   in that case.
\vskip .2cm
\noindent $\bullet$ {\it  Case 3.} $s=5$. Recall the relation (see \eqref{3fev1})
$$
N_{1}N_{2}N_{3}N_{4}N_{5}VW \asymp x.
$$
 We partition $(\mathcal M, \tilde{\mathcal N}, W)$ into
\begin{equation}\label{12:05}
(N_1VW)\geqslant N_2\geqslant N_3\geqslant N_4\geqslant N_5 (\geqslant x^{{1\over 6}-\eps}).
\end{equation}
Since $(\mathcal M, \tilde{\mathcal N}, W)$ is good, we have, by $(B_5)$, the inequality $N_3 >\Delta N_5$. We split the arguments in four cases.
\vskip .2cm
$\bullet$ {\it Case 3.1.}   $N_5N_4N_2 >Q \mathcal L^A$. As in \cite{BFI2}, we apply Proposition \ref{trilinear} with the conditions $(S2)$. The choices are $K=N_4$, $L=N_2N_5$ and $M=N_1VW N_3$.  We check the three inequalities of $(S2)$ as follows
$$
KL= N_5N_4N_2 >Q \mathcal L^A,$$
by hypothesis,
\begin{align*}
KL^2Q^2&= N_4(N_2N_5)^2 Q^2 < \Delta^{-1} N_5N_4N_3 N_2^2Q^2
\ll \Delta^{-1} xQ^2 \\
&\ll \Delta^{-1}x^2 y\ll x^2 y^{-3}\ll x^2\mathcal L^{-A},
\end{align*}
by the assumptions \eqref{16jan2} and \eqref{defDelta}.  Finally, we check
$$
K^2=N_4^2\ll (x/N_5)^{1\over 2} \ll  x^{{5\over 12}+{\eps\over 2}}\ll Qx^{-\eps}.
$$
It suffices to choose $A$ larger than $A_0 (=A_0 (B_7))$ as defined in Proposition \ref{trilinear} to conclude the proof of  \eqref{11:22} in that case.
\vskip .2cm
$\bullet$  {\it Case 3.2.}  $xQ^{-1} \mathcal L^{-A} < N_5N_4N_2 \leqslant Q\mathcal L^{ A}$.  Remark that the  last inequality implies  
$$  N_2\leqslant   x^{{1\over 6} + 3\eps}.$$
 As in \cite{BFI2}, we apply Proposition \ref{trilinear} with the conditions $(S1)$ with $K=N_2$ and $L=N_4N_3$ and $M=N_1VWN_5$. We check the three conditions of 
$(S1)$ as follows
$$
KL=N_4\,N_3\,N_2 > \Delta\, N_5\,N_4\,N_2 > \Delta\, x\,Q^{-1}  \mathcal L^{-A} >Q \mathcal L^A,
$$
the last inequality being a consequence of \eqref{16jan2} and \eqref{defDelta}. We also have
$$
K^2L^3 =N_2^2( N_4N_3)^3 \leqslant N_2^8  \ll   Qx\mathcal L^{-A},
$$
and
$$
K^4L^2 (K+L)\ll K^4 L^3 = N_2^4 (N_4N_3)^3\ll N_2^{10} \ll x^{{5\over 3}+ 30 \eps} <x^{2-\eps}.
$$
In that case also, we proved \eqref{11:22} by choosing $A$ larger than $A_0 =A_0 (B_7)$ as defined in Proposition \ref{trilinear}.

\vskip .2cm
$\bullet$  {\it Case 3.3.}  $x^{10\over 21} < N_5N_4N_2 \leqslant xQ^{-1} \mathcal L^{-A}$. This case does not appear in \cite{BFI2}. We appeal to  Proposition \ref{fouvrytrilinear}, with the choices  $L=N_1VW$, $M=N_5N_4N_2$ and $N=N_3$. We directly check
$$
LN= N_1VW N_3\gg x/( N_5N_4N_2)\gg Q\mathcal L^A,
$$
by assumption.
To check the conditions of $(S3)$, we first notice that the inequalities \eqref{12:05}  and $N_1N_2N_3 N_4 N_5 V W \sim x$ imply 
\begin{equation}\label{12:12}
N_1\leqslant x^{{1\over 3}+4\eps}\text{ and }N_3 \leqslant x^{{2\over 9}+\eps}.
\end{equation}
 It is easy to check that the first inequality
$
L^2N < M^{2-\eps} 
$
is satisfied when one has  $L^4N^3 < x^{2-\eps}$. But this  last inequality is true, since we write 
$$
L^4N^3 =L (LN)^3
\ll x^{{1\over 3} + 6 \eps}\ (x^{11\over 21} )^3 \ll  x^{{40\over 21} +6 \eps},
$$
 by \eqref{12:12} and the hypothesis of this case.

We now see that the second condition
of $(S3)$ $L^3N^4< M^{4-\eps}$ is satisfied when one has 
$L^7N^8< x^{4-\eps}$. To check this inequality, in our case, we write
$$
L^7N^8\ll  (LN)^7 N \ll (x^{11\over 21})^7 x^{{2\over 9}+\eps}\ll x^{{35\over 9}+\eps}\ll x^{4-\eps},
$$
by hypothesis of Case 3.3 and by \eqref{12:12}. 
The last condition of $(S3)$ is trivial. In conclusion \eqref{11:22} is also proved in that case, by choosing $A$ sufficiently large .
\vskip .2cm
$\bullet$ {\it Case 3.4.}  $N_5N_4N_2 \leqslant x^{10\over 21}$.  We appeal
to Proposition \ref{trilinear} (Conditions $(S1)$) with the choices
$K=N_3$, $L=N_1$. We verify each of these conditions by writing 
$$
KL =N_3N_1 \gg x/(VWN_5N_4N_2) \gg x^{11\over 21} x^{- 2\eps}\gg Q\mathcal L^{A_{0}},
$$
by hypothesis of Case 3.4,
$$
K^2L^3 =N_3^2 N_1^3 \ll x^{{13\over 9} +14\eps} \ll Qx \mathcal L^{-A_{0}},
$$ 
by \eqref{12:12} and, similarly
$$
K^4L^2 (K+L) \ll K^4L^3 =N_3^4 N_1^3\ll x^{{17\over 9} +16\eps}\ll x^{2-\eps}.
$$
The proof of \eqref{11:22} is now complete in this case. It follows that it is complete  in all the cases corresponding to $s=5$.

\vskip .2cm
\noindent $\bullet$ {\it Case 4}.
$s=4$. We start from the following configuration
\begin{equation}\label{14:56}
N_1\geqslant N_2\geqslant N_3\geqslant N_4 \ (\geqslant x^{{1\over 6}-\eps}).
\end{equation}
\vskip .2cm
$\bullet$ {\it Case 4.1.} We suppose that $N_4N_1 > x^{{1\over 2}+3\eps}$. As in \cite{BFI2}, we apply Proposition \ref{trilinearsmooth} with  $K=N_4$,  $L=N_1$ and $M=N_2N_3VW$. We check that
$$
KL=N_4N_1 > Q\mathcal L^A_{0}.$$
Using the trivial inequality $N_4 \ll x^{1\over 4}$, we check  the last two conditions of $(S4)$ by writing 
$$
MK^4Q \ll xN_4^3 N_1^{-1} Q\ll x^{{1\over 2}-3\eps}N_4^4 Q\ll x^{2-\eps},
$$
and
$$
MK^2Q^2\ll xN_4N_1^{-1} Q^2 \ll x^{{1\over 2}-3\eps} N_4^2 Q^2\ll x^{2-\eps}.
$$Here also \eqref{11:22} is proved in that case.

\vskip .2cm
$\bullet $ {\it Case 4.2.} We suppose that we have $N_4N_1 <x^{{1\over 2}+3\eps}$
But the following of our discussion will depend on the effect of the factor  $VW$ on the inequality \eqref{14:56}.
\vskip .2cm
$\bullet $ {\it Case 4.2.1.} We now suppose
$$
N_1 \geqslant VWN_2 \geqslant N_3\geqslant N_4\ ( > x^{{1\over 6}-\eps}).
$$
Since we are not in a boundary configuration, we have $N_1 > \Delta (VWN_2)$ or $N_3 >\Delta N_4$. We follow the technique of \cite{BFI2} by applying Proposition \ref{trilinear} (Conditions (S1)), with the choice $K=N_3$ and $L=N_1$. We check that
$$
KL =N_1N_3> (\Delta N_1(VWN_2) N_3 N_4)^{1\over 2} \gg (\Delta x)^{1\over 2} >x^{1\over 2} y^2 > Q \mathcal L^A.
$$
It remains to check the inequalities
\begin{equation}\label{08:45}
K^2L^3 < Qx\mathcal L^{-A},
\end{equation}
and
\begin{equation}\label{17:32}
  K^4L^2 (K+L)\ < x^{2-\eps}.
\end{equation}
It is easy to see that  both \eqref{08:45} \& \eqref{17:32} are  implied by the two inequalities
\begin{equation}\label{lecture}
K^2L^3 < x^{{3\over 2} -\eps} \ \&\  K^4L^3< x^{2-\eps}.
\end{equation} 
In other words, \eqref{11:22} is proved if we suppose the truth of \eqref{lecture}, always under the hypothesis of Case 4.2.1.

We now suppose  that  at least one of the inequalities of \eqref{lecture} is  not satisfied.  Then we turn our attention to Proposition \ref{trilinearsmooth}. To follow the notations of that Proposition, we  fix $K=N_3$,  $L=N_1$ and $M= N_2N_4VW$. In order to check the conditions $(S4)$, we easily see  that they are satisfied if one has
\begin{equation}\label{emergency}
 N_1N_3 >x^{{1\over 2}+\eps},\ N_3^3<N_1 x^{{1\over 2}- 2\eps}\text { and }
 N_3< N_1x^{-2\eps}.
 \end{equation}
 Also recall that the inequalities  $N_1N_2N_3N_4 <x$ and \eqref{14:56} imply
 \begin{equation}
 \label{emerg1}
 N_1N_3^2 \leqslant  x^{{5\over 6} +\eps}.
 \end{equation}

\smallskip
$\bullet$ Suppose that $K^2L^3 = N_1^3 N_3^2 \geqslant  x^{{3\over 2}-\eps}$. Then, we deduce that $N_1^3N_3^3 \geqslant (x^{{1\over 6}-\eps})\cdot (x^{{3\over 2} -\eps})$, which implies that  $N_1N_3 > x^{{1\over 2} +\eps}$. This is the first condition of \eqref{emergency}.

To check the  second condition, we combine with the square of \eqref{emerg1} to write
$
(N_1^3 N_3^2)\cdot (x^{{5\over 6}+\eps})^2 \geqslant x^{{3\over 2} -\eps} (N_1 N_3^2)^2$. This is equivalent to $N_1x^{{1\over 6}+3\eps} \geqslant N_3^2$. This certainly implies the second condition of  \eqref{emergency}since \eqref{emerg1} gives $N_3\leqslant x^{{5\over 18} +\eps}.$

For the last one, we start from $N_3 \leqslant x^{{5\over 18}+\eps}$. However, by hypothesis, we have $N_1 \geqslant (x^{{3\over 2} -\eps} N_3^{-2})^{1\over 3}\geqslant  x^{{17\over 54} -\eps}$. This gives the third inequality of \eqref{emergency} since $17/54 >5/18.$ 

 \smallskip 
 $\bullet$ Suppose that  $K^4L^3=N_1^3N_3^4 \geqslant x^{2-\eps}$. Then we deduce that $N_1^4N_3^4> (x^{2-\eps})\cdot (x^{{1\over 6}-\eps})$, which implies that  $N_1N_3> x^{{1\over 2}+\eps}$. This is the first condition of \eqref{emergency}.
 
 For the second condition, we raise the inequality  \eqref{emerg1} to the power $5/2$, so we have $(N_1^3 N_3^4)\cdot (x^{{5\over 6}+\eps})^{5\over 2} \geqslant  x^{2-\eps} (N_1N_3^2)^{5\over 2}$. This is equivalent to  $N_1 x^{{1\over 6} + 7\eps} \geqslant N_3^2$, from which we deduce the second inequality of \eqref{emergency} as a consequence of $N_3 \leqslant x^{{5\over 18}+\eps}$.  
 
 To check  the last inequality, we start from $N_3 \leqslant x^{{5\over 18}+\eps}$. However, by hypothesis, we have $N_1 \geqslant (x^{{  2} -\eps} N_3^{-4})^{1\over 3}\geqslant  x^{{8\over 27} -2\eps}$. This gives the third inequality of \eqref{emergency} since $8/27 >5/18.$ 
 
\vskip .2cm
$\bullet $ {\it Case 4.2.2.} We now investigate the following situation
$$
N_1 \geqslant N_2 \geqslant N_3\geqslant VWN_4\ ( > x^{{1\over 6}-\eps}).
$$
Since we are not in a boundary configuration, we have $N_1 > \Delta N_2$ or  $N_3> \Delta VWN_4.$ We apply the same technique as in Case 4.2.1, with $K=N_3$, $L=N_1$ and $M=N_2N_4 VW.$ The calculations are the same.

\vskip .2cm
$\bullet $ {\it Case 4.2.3.} We now suppose that we have both conditions
\begin{equation}\label{18:15}
N_1 \leqslant VWN_2\text{ and } N_3\leqslant VWN_4. 
\end{equation}
In other words, $N_1$ has almost the same order of magnitude as $N_2$, the same is true also for $N_3$ and $N_4$. Combining  \eqref{18:15} with the relation $N_1N_2N_3N_4VW \sim x$, we get 
\begin{equation}\label{18:08}
\Bigl({x\over 2 VW}\Bigr)^{1\over 2} \leqslant N_1N_3 \leqslant (xVW)^{1\over 2}.
\end{equation}
Consider the sequence $VWN_1\geqslant N_2\geqslant N_3 \geqslant N_4.$
From the inequalities $(VWN_1)\leqslant (VW)^2 N_2$  and $N_3 \leqslant VWN_4$, we deduce that, necessarily,  we have  $VW>\Delta^{1\over 2}$ otherwise, the above configuration would be a boundary configuration of type $(B_4)$.  We apply Proposition \ref{fouvrytrilinear} with 
$L=N_3$ and $N=N_1VW$. It easy to check that
$$
x\ll N_1N_2N_3N_4 VW \leqslant N_1^2 N_3^2VW =(LN)^2 (VW)^{-1}
$$
which implies that $LN >x^{1\over 2} (VW)^{1\over 2}\geqslant x^{1\over 2}\Delta^{1\over 4} > Q\mathcal L^A$, by the definition of $
\Delta$.

The two first inequalities of the set of conditions $(S3)$,    are consequences of   the inequalities $L^4N^3<x^{2-\eps}$ and $L^7N^8 < x^{4-\eps}$ (see discussion in Case 3.3). Hence, by \eqref{18:08},  we write
$$
L^4N^3=N_3^4 (N_1VW)^3= N_3 (N_1N_3)^3( VW)^3\ll x^{{3\over 2}+10\eps} N_3 < x^{{7\over 4}+11\eps},
$$
and
$$
L^7N^8 = N_3^7( N_1VW)^8 = N_1 (N_1N_3)^7 (VW)^8
\ll  x^{{7\over 2}+20\eps} N_1 \leqslant x^{{23\over 6}+25 \eps},
$$
since, by \eqref{18:15}, we deduce $N_1< x^{{1\over 3}+2\eps}$
and $N_3<x^{{1\over 4}+ \eps}$, from the inequality $N_1N_2N_3N_4\leqslant x$. The last inequality of $(S3)$ is trivially verified.

\vskip .2cm
\noindent $\bullet$ {\it Case 5.}  $s=1$, $2$ or $3$. Recall the inequalities \eqref{assumption1}. For $s=3$, the corresponding sum $\mathcal F_j (\mathcal M, {\mathcal N}, U, W,Q)=\mathcal F_j (\mathcal M, \tilde{\mathcal N},  W,Q)$ satisfies the inequality
\begin{align*}
\mathcal F_j (\mathcal M, {\mathcal N}, &U, W,Q)\leqslant 
\sum_{q\sim Q\atop (q,a)=1} \sum_{v\leqslant 2V\atop (v,q)=1}
\sum_{w\simeq W\atop (w,q)=1} \\
&\Bigl\vert \underset{n_1n_2n_3\equiv a\overline v\overline w\bmod q\atop n_1\simeq N_1,\, n_2 \simeq N_2,\, n_3 \simeq N_3}{\sum\ \sum\ \sum}
\mathfrak z(n_1n_2n_3)
-{1\over \varphi (q)} \underset{(n_1n_2n_3,q)=1 \atop n_1\simeq N_1,\, n_2 \simeq N_2,\, n_3 \simeq N_3}{\sum\ \sum\ \sum}
\mathfrak z(n_1n_2n_3)\ 
\Bigr\vert.
\end{align*}
Since we have $N_1N_2N_3>x^{1-2\eps}>Q^{1\over 1/2+\delta}$, we apply Proposition \ref{deligne} to the expression inside $\vert \cdots \vert$ giving
\begin{align*}
\mathcal F_j (\mathcal M, {\mathcal N}, U, W,Q)&\ll 
\sum_{q\sim Q\atop (q,a)=1} \sum_{v\leqslant 2V\atop (v,a)=1}
\sum_{w\simeq W\atop (w,a)=1} {N_1N_2N_3\over \varphi (q)}\cdot \exp\bigl( -{\eps\over 2} (\log \log x)^2\bigr)\\
&\ll x\mathcal L^{-B_7},
\end{align*}
since $N_1N_2N_3VW \ll x.$ This gives \eqref{11:22}. The case $s=2$ is treated in a similar way. The case $s=1$ is trivial.

 This completes the proof of Theorem \ref{extension}.

  \section{ Proof of Theorem \ref{L--V}}\label{Thm1} Theorem \ref{L--V} is an easy consequence of Proposition \ref{fouv1}.
  \begin{proof} Under the assumptions of Theorem \ref{L--V},  let
  $$
  Y=N\Bigl( 1 -{1\over  2(\log N)^A}\Bigr),   Z_1= N^2\Bigl( 1 -{1\over  2(\log N)^A}\Bigr)\text { and } Z_2= N^2\Bigl( 1 -{1\over   (\log N)^A}\Bigr).
  $$  
  Now consider  the quantity
  $$
S:=  \sum_{Y \leqslant a \leqslant N} \Bigl( \pi (Z_1; a, 1) -\pi (Z_2; a, 1)\Bigr).
  $$   
  This sum is counting (with multiplicities of representations) the primes 
 $ p$ satisfying  $ Z_2 <p\leqslant Z_1$  congruent to $1\bmod a$, with $Y\leqslant a \leqslant N$. Hence, such $p$ are of the form $p=1+ab$ with  $a\leqslant N$ and $b\leqslant    (Z_1-1)/Y \leqslant N.$ 
 
 Proposition \ref{fouv1}  and the Prime Number Theorem give  the relations 
 \begin{align*}
 S&=\bigl(\pi (Z_1)-\pi (Z_2)\bigr)\ \sum_{Y\leqslant a \leqslant N}  {1\over \varphi (a) } + O\bigl( N^2 (\log N)^{-2A-2}\bigr)\\
&\geqslant    \bigl(\pi (Z_1)-\pi (Z_2)\bigr)\cdot    {N-Y\over N} - O\Bigl( {N^2\over (\log N)^{2A+2}}\Bigr)\\
&\geqslant  { N^2\over 8(\log N)^{2A+1} }(1-o(1)) -O\Bigl({N^2\over (\log N)^{2A+2}}\Bigr)\\
&\geqslant 1,
 \end{align*}
 for $N\geqslant N_0 (A)$. 
 Hence, the sum $S$ is not empty, this implies the existence of a $p$ with the required property.
  \end{proof}
 \section{Proof of  Theorem \ref{firstgeneral}}\label{Thm2}
 This proof has many similarities with the proof of Theorem \ref{L--V}. The main difference, is that we shall appeal to Proposition \ref{Q=X^1/2} instead of Proposition \ref{fouv1}.  We are searching for primes of the form $p=ab+1$, with    $b\in \mathcal B$,   $a,\,b\leqslant N$, and $(1-2\delta) N^2   <p\leqslant (1- \delta)N^2 $.  This set of primes $p$ certainly contains the set
 \begin{align*}
\mathcal E (\mathcal B, N):= \bigl\{ p\ &;\  (1-2\delta )N^2  <p\leqslant (1-\delta )N^2  ,\\
 &\ p\equiv 1   \text{ modulo some }  
b\in \mathcal B \text{ satisfying }  (1-\delta) N < b \leqslant N\bigr\},
 \end{align*}
and $\mathcal E (\mathcal B, N)$  is non empty if and only if the sum $S_1$ defined by
$$
S_1:= \sum_{b\in \mathcal B\atop (1-\delta) N< b \leqslant N}\Bigl\{ \pi \bigl((1-\delta )N^2 ; b,1\bigr) -
 \pi \bigl((1-2\delta )N^2  ; b,1\bigr)\Bigr\},
$$
is non zero.
 By  \eqref{2jan1} deduced from Proposition  \ref{Q=X^1/2}, we see  that  $S_1$ satisfies the equality
 \begin{equation}\label{17:21}
 S_1= \Bigl \{\pi \bigl((1-\delta)N^2  )-\pi ((1-2\delta)N^2 \bigl)\Bigr\}\  \sum_{b\in \mathcal B\atop (1-\delta) N <b\leqslant N} {1\over \varphi (b)} +O \Bigl( {N^2\over \log^3 N}\cdot (\log \log N)^{B_1}\Bigr).
  \end{equation}
 By the Prime Number Theorem, we know that the term inside $\{\cdots \}$  in \eqref{17:21} is $\sim \delta N^2/(2\log N)$, as $N\rightarrow \infty$.  
 Inserting this  formula  into \eqref{17:21} and using the assumption \eqref{7jan1},  we obtain $S_1\geqslant 1$, by choosing $c_6=B_1+1$ and $N\geqslant c_7 (\delta)$. This completes the proof of Theorem 
 \ref{firstgeneral}.
 \section{Proof of  Theorem \ref{secondgeneral}}  Both   sequences  $\mathcal A$ 
 and $\mathcal B$ are now quite general  and    to shorten notations, we write
$$
A:=\vert \mathcal A\vert\text{ and } B:=\vert \mathcal B\vert.
$$
  Hence we have the inequalities
  \begin{equation}\label{08:51}
A\geqslant  B \geqslant  N/(\log N)^{\delta}\text{ with } 0<\delta <1,
 \end{equation}
as a consequence of \eqref{lowerbound}.   We are using the Tchebychev--Hooley method as it is done in \cite{SaSt} \& \cite{St}.  Consider the product
 \begin{equation}\label{defE}
 E= E(\mathcal A, \mathcal B, N):=\underset {a\in \mathcal A, \ b\in \mathcal B \atop a,\, b \leqslant N} {\prod\  \prod }(ab+1).
 \end{equation}
Taking logarithms, we deduce the lower bound
$$
  \log E \geqslant \underset{a\in \mathcal A\ b\in \mathcal B}{\sum \ \sum} \log (ab)\geqslant \underset{a\leqslant  A\ b\leqslant B}{\sum \ \sum} \log (ab),
$$
  which leads  to the lower bound
   \begin{equation}\label{12jan10}
   \log E\geqslant (2-o_{\delta }(1)) AB\log N,
   \end{equation}
uniformly under the condition   \eqref{08:51}.

  Let 
 $$
  P=P^+ (E),
 $$
 be the number for which we are looking for a lower bound in terms of $A$,
  $B$ and $N$.
As in \cite[\S 6]{St}, we introduce 
 $$
 E_1 =\prod_{p\leqslant N} p^{v_p ( E)},
 $$
 where $v_p$ is the $p$--adic valuation.  We first prove (compare with \cite [Lemma 3]{St})
 \begin{proposition}\label{12jan1} 
 For every subsets $\mathcal A$ and $\mathcal B$ of $[1,\dots, N]$, with cardinalities $A$ and $B$ satisfying \eqref{08:51}, we have the inequality
      \begin{equation}\label{31jan00}
  \log E_1 \leqslant (1+o_{\delta }(1))\, AB \log N.
  \end{equation}
as $N\rightarrow \infty$.
 \end{proposition}
 \begin{proof} We first recall the following lemma (see \cite[Lemma 4]{SaSt})
 \begin{lemma}\label{squareoferrors}
 Let $N$ be a positive integer and $\mathcal U \subset [1,\dots, N]$. Let $h$ and $m$ be integers with $m\geq1$ and let 
 $$
 r(\mathcal U,h,m):= \bigl\vert \{ u\in \mathcal U\, ;\ u\equiv h\bmod m\}\bigr\vert.
 $$
 We then have the inequality
 $$
 \sum_{p\leqslant N} \log p \sum_{k\leqslant  {\log N\over \log p}}\  \sum_{h=1}^{p^k} \bigl( r(\mathcal U, h,p^k)\bigr)^2 
 \leqslant 
   \vert \mathcal U\vert \ \bigl( \vert \mathcal U\vert -1 +\pi (N) \bigr) \log N.$$
 \end{lemma}
 The proof of Proposition \ref{12jan1} now follows the proof of \cite [4.14]{SaSt}. We have
\begin{align}
 \log E_1 &= \sum_{p\leqslant N} v_p (E) \log p\nonumber\\
 & =\sum_{p\leqslant N} \log p  
 \sum_{k\leqslant { \log  (N^2+1)\over  \log  p}} \ 
 \Bigl\vert
 \bigl \{ 
 (a,b)\in \mathcal A \times \mathcal B\,;\ ab\equiv -1 \bmod p^k
 \bigr\}
 \Bigr\vert \nonumber\\
 &= \Sigma_1 +\Sigma_2,\label{11jan1}
 \end{align}
 where 
 $\Sigma_1$ and $\Sigma_2$ respectively correspond to the cases $1\leqslant k\leqslant{\log N \over \log p}$ and $ {\log N\over \log p} < k\leqslant {\log (N^2+1)\over \log p}$.
Denoting by $\overline h$ the multiplicative inverse of $h\bmod p^k$ and  using the  Cauchy--Schwarz inequality,  we get 
\begin{align}
& \Sigma_1 
 =
 \sum_{p\leqslant N} \log p  
 \sum_{k\leqslant { \log   N\over  \log  p}}  
 \sum_{1\leqslant h \leqslant p^k\atop (h,p^k)=1}r(\mathcal A, h,p^k) r(\mathcal B, -\overline h,p^k)\nonumber\\
 &\leqslant\Bigl\{
 \sum_{p\leqslant N} \log p  
 \sum_{k\leqslant { \log   N\over  \log  p}}  
 \sum_{1\leqslant h \leqslant p^k\atop (h,p^k)=1} r^2(\mathcal A, h,p^k)\Bigr\}^\frac{1}{2}\cdot   \Bigl\{
 \sum_{p\leqslant N} \log p  
 \sum_{k\leqslant { \log   N\over  \log  p}}  
 \sum_{1\leqslant h \leqslant p^k\atop (h,p^k)=1} r^2(\mathcal B, h,p^k)\Bigr\}^\frac{1}{2} \nonumber\\
 &\leqslant  \Bigl\{   A  \bigl(   A-1 +\pi (N) \bigr)\Bigr\}^\frac{1}{2} \cdot \Bigl\{   B  \bigl(   B-1 +\pi (N) \bigr)\Bigr\}^\frac{1}{2}(\log N)\nonumber\\
 &\leqslant (1+o(1))\,AB\, \log N,
  \label{11jan2}
\end{align}
the last lines being a consequence of Lemma \ref{squareoferrors} and of the condition \eqref{08:51}.

For $\Sigma_2$, we remark that we have the inequality $p^k>N$, hence when $b$ is fixed, the equation $ab\equiv -1 \bmod p^k$ has at most one solution in $a$. From this, we  deduce
the inequality
\begin{align}\label{11jan3}
\Sigma_2 &
\leqslant 
\sum_{p\leqslant N} \log p \ \sum_{{\log N \over \log p}< k \leqslant {\log (N^2+1)\over \log p} }
  B \nonumber\\
  & \leqslant   B    \sum_{p\leqslant N} \log (N^2+1) =   B \cdot  \pi (N) \cdot \log (N^2+1).
\end{align}
Putting together \eqref{08:51},  \eqref{11jan1}, \eqref{11jan2} and \eqref{11jan3}, we complete the proof  of Proposition \ref{12jan1}.
 \end{proof}
 
 \subsection{Use of Theorem \ref{extension}} \label{6:1} Let
 $$
 E_2=\prod_{N< p\leqslant P\atop p \mid E}   p^{v_p(E)},
 $$
where $E$ is defined by \eqref{defE}.  We then have
 $$
 \log E_2 =\log E -\log E_1.
 $$
   By \eqref{12jan10}  \& \eqref{31jan00} we have the lower bound  
\begin{equation}\label{12jan11}
 \log E_2 \geqslant (1-o_{\delta }(1))\, AB\, \log N.
\end{equation}
 We are now searching for an upper bound of $\log E_2$.  Since $v_p (ab+1)\leqslant 1$ for any $p$ satisfying $N<p\leqslant P$ and $a$, $b\leqslant N$, we have the equality
$$
v_p(E_2)=\bigl\vert\bigl\{ (a,b)\in \mathcal A\times \mathcal B\,;\  ab+1\equiv 0\bmod p
\big\}\bigr\vert. 
$$
We now use  the property   $\mathcal A \subset [1,\dots,N]$ to write the  inequality
\begin{equation}\label{12jan12}
\log E_2 \leqslant \sum_{(a,b,m,p)\in \mathcal Q} \log p,
\end{equation}
where the sum is over the set  $\mathcal Q$  of  the quadruples $(a,b,m,p)$ defined by
$$
\mathcal Q :=\bigl\{ (a,b,m,p)\,;\ ab+1=pm,\  N<p\leqslant P, \ 1\leqslant a,\, b \leqslant N,\ b\in \mathcal B\bigr\}.
$$
 We want to drop the inequality $a\leqslant N$ in order to apply  Theorem  \ref{extension}. So we include $\mathcal Q$ in the disjoint  union
$$
\mathcal Q \subset \bigcup_{\ell = 0}^{\ell_0}{\mathcal R}_\ell,
$$
where 
\begin{align*}
{\mathcal R}_\ell :=\Bigl\{ (a,b,m,p)\,;\ ab+1=pm&,\ pm\leqslant (1-\kappa)^\ell N^2+1,\  N<p\leqslant P, \\  
& b\in \mathcal B,\ (1-\kappa)^{\ell +1 } N< b \leqslant (1-\kappa)^{\ell  }N\Bigr\},
\end{align*}
where   
 $\kappa =\kappa (N)$ is a function of $N$ tending to zero as $N$ tends to $\infty$ and $\ell_0$ is the integer defined by $(1-\kappa)^{\ell_0 +1}N < 1\leqslant (1-\kappa)^{\ell_0}N.$  This integer $\ell_0$ satisfies
 $
 \ell_0=O (  \kappa^{-1}\log N).$
 Using this decomposition, we transform \eqref{12jan12} into the inequality
\begin{equation}\label{14jan10}
\log E_2 \leqslant
\sum_{\ell =0}^{\ell_0}\ T_\ell,
\end{equation}
with
$$
T_\ell:= \sum_{b\in \mathcal B\atop (1-\kappa)^{\ell +1} N< b \leqslant (1-\kappa)^\ell N}\  \underset{N<p\leqslant P,\ pm\leqslant (1-\kappa )^\ell N^2+1 \atop pm\equiv 1\bmod b}{\sum\ \ \sum}  \log p.
$$ Note that we lose a lot of information over $a \in \mathcal A$ when replacing the equality $pm =1 +ab$ by the congruence condition $pm \equiv 1 \bmod b$.
Since we always have the inequality $\{(1-\kappa)^\ell N\}^2\leqslant 3 \bigl( (1-\kappa)^\ell \,N^2+1)   $ for $\ell \leqslant  \ell_0$,  we are now in good position to apply   Theorem \ref{extension} (with $y=3$) to $T_\ell$. Let $\rho (N^2): =N^2(\log \log N)^{B_2}(\log N)^{-1}$. With  this theorem   we have the equality 
\begin{align}\label{music}
T_\ell &= \sum_{b\in \mathcal B\atop (1-\kappa)^{\ell +1} N< b \leqslant (1-\kappa)^\ell N}\  {1\over \varphi (b)}\  \Bigl\{\underset{N<p\leqslant P,\ pm\leqslant (1-\kappa )^\ell N^2+1 \atop (pm,b)= 1 }{\sum\ \ \sum}  \log p\Bigr\} +O\bigl( \rho ((1-\kappa)^\ell N^2+1)\bigr)\nonumber\\
&=T_{\ell}^{(1)}+O\bigl( T_{\ell}^{(2)}\bigr),
\end{align}
by definition.  Using the equality  \eqref{31jan10} and the Prime Number Theorem, we transform $T_{\ell}^{(1)}$ as follows
 \begin{align}
T_{\ell}^{(1)}&=\sum_{b }\ {1\over \varphi (b)}\  \Bigr\{ \sum_{N<p\leqslant P}  \log p \Bigl( {\varphi (b)\over b}\cdot {(1-\kappa)^\ell N^2\over p} +O(\tau (b))\Bigr)\Bigr\}\nonumber   \\
 &=N^2\cdot \Bigl(\log{P\over N}+o(1)\Bigr) \cdot \Bigl( (1-\kappa)^\ell  \sum_{b     }\ {1\over b} 
\Bigr)  
 +O\Bigl( P \sum_b {\tau (b) \over \varphi (b)}\Bigr)\nonumber
\\
&\leqslant  N^2\cdot \Bigl(\log {P\over N}+o(1))\Bigr)\cdot  \Bigl( {1\over (1-\kappa) N}\cdot \sum_{b }\ {1 } \Bigr)   +O\Bigl( P \sum_b {\tau (b) \over \varphi (b)}\Bigr).\label{music1}
\end{align}
 In \eqref{music1}, the conditions of summation are always:  $b\in {\mathcal B} $ and  $(1-\kappa)^{\ell +1} N< b \leqslant (1-\kappa)^\ell N$. 
We now sum over all the $\ell \leqslant \ell_0$. We first write that
$$
\sum_{\ell =0}^{\ell_0} T_{\ell}^{(2)}\ll      {N^2\over  \kappa \log N}\cdot (\log \log N)^{B_2}.
$$
Combining this relation with   \eqref{14jan10}, \eqref{music} \& \eqref{music1}, and summing $T^{(1)}_{\ell}$ over $\ell$, we have the inequality
$$
\log E_2 \leqslant {BN\over 1-\kappa}\cdot \Bigl(\log {P\over N}+o(1)\Bigr) +O\Bigl(  {N^2\over  \kappa \log N}\cdot (\log \log N)^{B_2}\Bigr) + O\bigl( P\log ^2N \bigr).
$$
Choosing $\kappa =(\log N)^{\delta-1\over 2}$  and recalling  \eqref{08:51} we obtain 
$$
\log E_2 \leqslant \bigl(1+o_\delta (1)\bigr)\cdot \Bigl(\log {P\over N}+o(1)\Bigr)\cdot BN.$$
 Comparing with \eqref{12jan11}, we obtain  the inequality
 $$ \log \frac{P}{N} \geqslant (1-o_\delta(1))\frac{A}{N} \log N-o(1)\geqslant (1-o_\delta(1))\frac{A}{N}\log N.
 $$
 This complets the proof of Theorem \ref{secondgeneral}.

\end{document}